\documentclass[11pt]{article}
\usepackage[T1]{fontenc}
\usepackage{lmodern} %
\usepackage{microtype} %
\usepackage{amsmath}
\usepackage{amssymb}
\usepackage{mathrsfs}
\usepackage[all]{xy} %
\usepackage{mathtools} %
\usepackage{tabularray} %
\usepackage{xparse} %
\usepackage{amsthm} %
\usepackage{thmtools} %
\usepackage{setspace} %

\usepackage{tcolorbox} %
    \tcbuselibrary{skins, theorems, breakable, hooks}

\usepackage{varwidth} %

\usepackage{xcolor} %
    \definecolor{jd_green}{HTML}{068f89}
    \definecolor{jd_blue}{HTML}{0675BB}
    \definecolor{jd_red}{HTML}{E03E52}
    \definecolor{jd_purple}{HTML}{957FBB}

\usepackage{enumitem} %
    \setenumerate{label=(\roman*), noitemsep, topsep=0pt}

\usepackage{needspace} %

\usepackage{tikz} %
    \usetikzlibrary{positioning}
    \usetikzlibrary{shadows.blur}

\usepackage{tocloft}
\newcommand*{\printtableofcontents}{{
    \vspace*{-1.5em}
    \hypersetup{linkcolor=black} %
    \tableofcontents
}}

\usepackage{glossaries-extra}

\immediate\write18{makeindex build/\jobname.glo -t build/\jobname.glg -s build/\jobname.ist -o build/\jobname.gls}

\GlsXtrEnableEntryCounting{abbreviation}{1}
\glsdisablehyper %

\newabbreviation{ym}{\text{YM}}{\text{Yang-Mills}}
\newabbreviation{sd}{\text{SD}}{\text{self-dual}}
\newabbreviation{asd}{\text{ASD}}{\text{anti-self-dual}}
\newabbreviation{pym}{\text{PYM}}{\text{primitive Yang-Mills}}
\newabbreviation{tym}{\text{TYM}}{\text{trace Yang-Mills}}

\usepackage[
    bookmarksdepth=subsubsection,
]{hyperref}
\usepackage[open, openlevel=2]{bookmark}

\hypersetup{
    hypertexnames = false,
    bookmarksnumbered = true,
    colorlinks = true,
    linktoc = all,
    citecolor = jd_green, %
    citebordercolor = jd_green, %
    linkcolor = jd_green,
    linkbordercolor = jd_green,
    urlcolor = jd_green,
    urlbordercolor = jd_green,
}

\usepackage[capitalise, noabbrev]{cleveref}
\crefformat{equation}{\textup{(#2#1#3)}}
\Crefformat{equation}{\textup{(#2#1#3)}}

\crefrangeformat{equation}{\textup{(#3#1#4)} to \textup{(#5#2#6)}}
\Crefrangeformat{equation}{\textup{(#3#1#4)} to \textup{(#5#2#6)}}

\crefmultiformat{equation}{\textup{(#2#1#3)}}{ and \textup{(#2#1#3)}}{, \textup{(#2#1#3)}}{, and \textup{(#2#1#3)}}
\Crefmultiformat{equation}{\textup{(#2#1#3)}}{ and \textup{(#2#1#3)}}{, \textup{(#2#1#3)}}{, and \textup{(#2#1#3)}}

\numberwithin{equation}{section}

\allowdisplaybreaks

\newcommand*{\defterm}[1]{\textbf{#1}}

\newcommand{\orcid}[1]{%
    \href{https://orcid.org/#1}{%
        \,\raisebox{-1pt}{\includegraphics[height=0.85em]{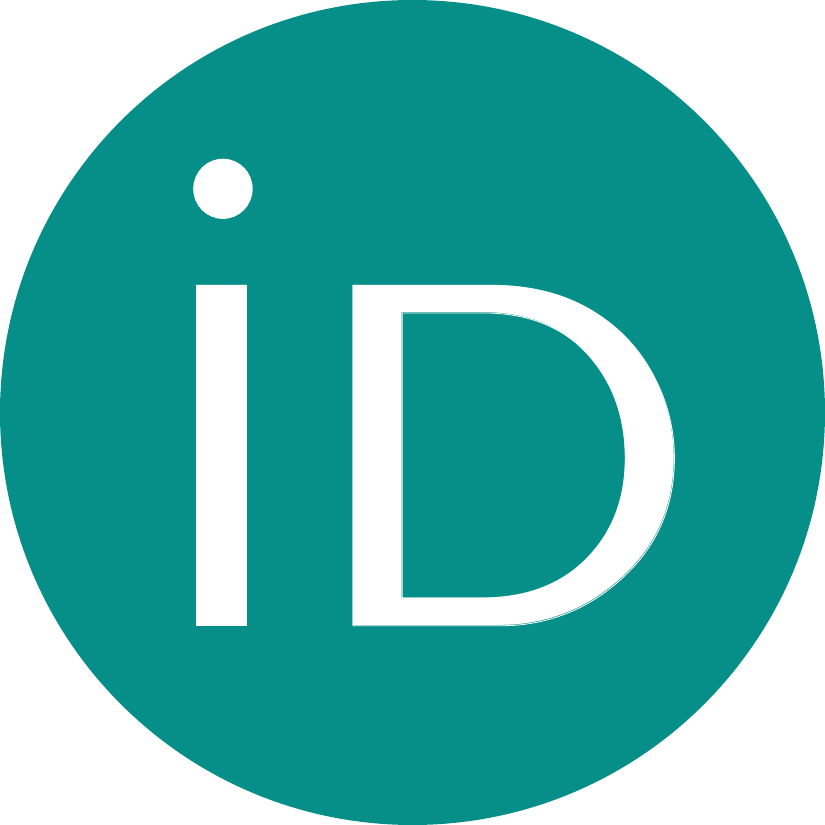}}%
    }%
}

\newcommand*{\affiliation}[2]{%
    \AtEndDocument{%
        \needspace{2\baselineskip}%
        {\small\noindent #1\\%
        \textit{Email address:}
        \href{mailto:#2}{\ttfamily #2}%
        \vspace{0.5em}}%
    }
}

\newcommand*{\laplace}{\Delta} %

\NewDocumentCommand{\CYM}{ O{\ensuremath{\zeta}} }{\operatorname{CYM}_{#1}}

\NewDocumentCommand{\eqdef}{ s }{
    \IfBooleanTF{#1}
    {\eqqcolon}
    {\coloneqq}
}

\newcommand*{\delp}{\partial_{+}}
\newcommand*{\delm}{\partial_{-}}
\newcommand*{\delpa}{\partial_{+A}}
\newcommand*{\delma}{\partial_{-A}}

\newcommand*{\Z}{\mathbb{Z}} %
\newcommand*{\R}{\mathbb{R}} %

\newcommand*{\calA}{\mathcal{A}} %
\newcommand*{\calJ}{\mathcal{J}} %
\newcommand*{\calL}{\mathcal{L}} %

\def\om{\omega}
\def\Om{\Omega}

\def\w{\wedge}

\DeclareMathOperator{\ad}{ad} %
\DeclareMathOperator{\SU}{SU} %
\DeclarePairedDelimiter{\norm}{\lVert}{\rVert} %

\newcommand{\cconj}[1]{\mkern2mu\overline{\mkern-2mu#1}} %
\NewDocumentCommand{\del}{ s }{\partial\IfBooleanT{#1}{{^*}}}
\NewDocumentCommand{\delbar}{ s }{\cconj{\del}\IfBooleanT{#1}{{^*}}}

\DeclareMathOperator*{\Diff}{Diff} %
\DeclareMathOperator{\SO}{SO} %
\NewDocumentCommand{\pushf}{ o o m }{ %
    \IfValueTF{#1}{ #3{}_{*,#1} }{ #3{}_{*} }
}
\newcommand*{\pullb}[1]{ #1^{*} } %
\newcommand*{\dd}{d} %
\newcommand*{\dx}{\dd x}
\newcommand*{\ddt}{\frac{\dd}{\dd t}}
\newcommand*{\vol}{\operatorname{vol}} %
\DeclarePairedDelimiter{\paren}{(}{)} %
\DeclarePairedDelimiter{\set}{\{}{\}} %
\DeclarePairedDelimiter{\abs}{\lvert}{\rvert} %
\DeclarePairedDelimiter{\innerp}{\langle}{\rangle} %
\DeclareMathOperator{\trace}{tr} %

\NewDocumentCommand{\Linnerp}{ m }{\paren*{#1}} %

\csundef{brack} %
\DeclarePairedDelimiter{\brack}{[}{]} %
\renewcommand*{\div}{\operatorname{div}} %

\NewDocumentCommand{\correspondencebox}{ > { \TrimSpaces } m }{
    \set*{
        \tcboxmath[
            blankest,
            top=4pt, bottom=4pt,
            left=2pt, right=2pt
        ]{
            \begin{varwidth}{0.5\textwidth}\centering
                #1
            \end{varwidth}
        }
    }
}

\newcommand*{\correspondence}[2]{
    \correspondencebox{#1} \longleftrightarrow \correspondencebox{#2}
}

\newcommand{\critYM}{\calA_{\operatorname{YM}}}
\newcommand{\critPYM}{\calA_{\operatorname{PYM}}}
\newcommand{\critTYM}{\calA_{\operatorname{TYM}}}

\theoremstyle{plain}
\newtheorem{thm}{Theorem}[section]
\crefname{thm}{Theorem}{Theorems}
\Crefname{thm}{Theorem}{Theorems}

\newtheorem{prop}[thm]{Proposition}
\crefname{prop}{Proposition}{Propositions}
\Crefname{prop}{Proposition}{Propositions}

\newtheorem{lem}[thm]{Lemma}
\crefname{lem}{Lemma}{Lemmas}
\Crefname{lem}{Lemma}{Lemmas}

\newtheorem{cor}[thm]{Corollary}
\crefname{cor}{Corollary}{Corollaries}
\Crefname{cor}{Corollary}{Corollaries}

\theoremstyle{definition}

\crefname{defn}{Definition}{Definitions}
\Crefname{defn}{Definition}{Definitions}

\newtheorem{rmk}[thm]{Remark}
\crefname{rmk}{Remark}{Remarks}
\Crefname{rmk}{Remark}{Remarks}

\newtheorem{ex}[thm]{Example}
\crefname{ex}{Example}{Examples}
\Crefname{ex}{Example}{Examples}

\newenvironment{acknowledgments}{%
    \pdfbookmark[2]{Acknowledgments}{acknowledgments}%
    \vspace{\baselineskip}%
    \noindent\textbf{Acknowledgments.}%
}{}

\AtEndDocument{
    \setstretch{1}
    \phantomsection
    \addcontentsline{toc}{section}{\refname}
    \bibliography{references.bib}
    \bibliographystyle{amsalpha}
}

\title{%
    Symplectic Yang-Mills Theory %
}
\date{}

\author{
    Jonathan Delgado\orcid{0009-0005-8120-8499},
    Li-Sheng Tseng\orcid{0000-0002-8189-6579},
    and
    Jiawei Zhou\orcid{0009-0004-5591-734X}
}

\affiliation %
{Department of Mathematics, University of California, Irvine, CA 92697, USA}
{jonathan.delgado@uci.edu}

\affiliation
{Department of Mathematics, University of California, Irvine, CA 92697, USA}
{lstseng@uci.edu}

\affiliation
{Department of Mathematics, Nanchang University, Nanchang, Jiangxi, 330031, China}
{jiaweizhou90@ncu.edu.cn}

\begin{document}
\maketitle

\begin{abstract} 
On a symplectic manifold, any differential two-form has a natural decomposition into two components: a primitive part and a non-primitive one.  Applying this decomposition to the curvature two-form of a principal bundle over a symplectic manifold, we obtain a natural splitting of the \gls{ym} functional into two functionals that intrinsically depend on the symplectic structure: the \gls{pym} functional and the \gls{tym} functional. We work out the basic properties of the critical solutions of these two functionals. The \gls{pym} functional in particular exhibits many of the desirable properties of the \gls{ym} functional, including the ellipticity of its Euler-Lagrange equations and an algebraic classification of its flat solutions on $G$-bundles.  We also prove a monotonicity formula for the \gls{pym} functional as a first step towards characterizing its moduli space of solutions.

\end{abstract}

\pdfbookmark[1]{Contents}{contents}%
\printtableofcontents

\section{Introduction}
\label{sec:introduction}

\glsfmtfull{ym} theory has played a profound role in the interaction between theoretical physics and pure mathematics for over half a century.  In differential geometry, when considering a principal bundle $P$ (or equivalently, a vector bundle $E$) over a Riemannian manifold $(M, g)$, it is standard to introduce a connection 1-form, $A$, to specify how to transport vectors over $M$.  On the space of connection forms, there is a simple functional which assigns to each connection form the norm-squared of the associated curvature two-form,
$F^A = \dd A + \frac{1}{2}\brack*{A, A}$.
This is just the \gls{ym} functional,
\begin{align*}
    \norm*{F^A}^2
    \eqdef \int_M \abs*{F^A}^2 \vol_g\,.
\end{align*}

In this paper, we consider the \gls{ym} functional when the base manifold $M$ has additionally a symplectic structure, i.e., $(M^{2n}, \om)$.   Since the symplectic structure $\om$ is a non-degenerate two-form, any differential $k$-form, $\alpha_k\in \Om^k(M)$, can be expressed as a polynomial in terms of powers of $\om$: 
\begin{align}\label{Ldef}
    \alpha_k = \beta_k +  \beta_{k-2} \w \om + \beta_{k-4} \w \om^2 + \ldots\,.
\end{align} 
This decomposition is typically referred to as the Lefschetz decomposition of differential forms on a symplectic manifold.  The differential forms
$\{\beta_k, \beta_{k-2}, \beta_{k-4}, \ldots\}$ which represent the ``coefficients'' of the polynomial expansion in $\om$ are called \textit{primitive} forms.  Heuristically, primitive forms can be thought of as those that cannot be written in terms of any non-trivial powers of $\om$.  The primitive coefficients in \cref{Ldef} can be computed explicitly; they are uniquely defined for a fixed $\alpha_k$ and $\om$ and generally vary with $\om$.
(For more details, see, for example, \cite[Section~2.1]{tsengCohomologyHodgeTheory2012II}.) 

That the curvature $F^A$ is a differential two-form suggests that we should apply the Lefschetz decomposition to it when $M$ is symplectic.  This results in two distinct terms:
\begin{align}\label{FAdecomp}
F^A = F^A_p + \Phi^A \om\,,
\end{align}
where $F^A_p$ denotes the primitive two-form component in the decomposition of \cref{Ldef}, and $\Phi^A \om$, being proportional to $\om$, is the non-primitive component with $\Phi^A$ a zero-form.  If, additionally, we work with a metric $g$ that is compatible with the symplectic structure $\om$, then the Yang-Mills functional also decomposes into two parts:
\begin{align*}
     \norm*{F^A}^2=\norm*{F^A_p}^2 + \norm*{\Phi^A\omega}^2\,.
\end{align*}
The first term, consisting of only the primitive component $F^A_p$ of the curvature,
\begin{align}
    \norm*{F^A_p}^2
    \eqdef \int_M \abs*{F^A_p}^2 \vol_g\,,
\end{align}
is what we shall call the \textbf{\glsfmtfull{pym} functional}. 
And we shall call the second term, consisting of only the non-primitive component $\Phi^A\omega$ of the curvature,
\begin{align}
    \norm*{\Phi^A\omega}^2 \eqdef
    \int_M \abs*{\Phi^A\omega}^2 \vol_g\,,
\end{align}
the \defterm{\glsfmtfull{tym} functional}.  Let us note that if $\dim M = 2$, then any two-form must be proportional to $\om$.  In this case, $F^A_p = 0$ and this implies the triviality of the \gls{pym} functional and that the \gls{tym} functional is just the two-dimensional \gls{ym} functional.  Because of this, we will assume in this paper that the manifold dimension $\dim M \geq 4$.

Hence, on a symplectic manifold $(M^{2n}, \om)$ for $n\geq 2$, it is natural when studying Yang-Mills theory to also consider  the \gls{pym} and \gls{tym} functionals. By definition, the \gls{pym} and \gls{tym} functionals are both intrinsically dependent on the symplectic structure and so studying them can inform us about the choice of the symplectic structure on $M$.

In studying these functionals, the main question that we pursue is to understand the space of critical solutions of the \gls{pym} and \gls{tym} functionals and compare their solution spaces with those of the \gls{ym} functional. The critical solutions satisfy the Euler-Lagrange equations of the functional, which for all three can be expressed concisely in terms of the adjoint of the covariant derivative, $\dd^*_A$, as can be seen in \cref{table:critical_points_of_functionals}.  Since we will be discussing solutions of three different functionals, we will say that $A$ is \defterm{Yang-Mills}, \defterm{primitive Yang-Mills}, or  \defterm{\glsfmtlong{tym}}, if $A$ is a critical solution of the \gls{ym}, \gls{pym}, or \gls{tym} functional, respectively.

\begin{table}[t]
    \centering
    \begin{tblr}{
        columns = {halign = c},
        row{1} = {font=\bfseries},
        cell{2-Z}{1-Y} = {mode = math},
        hline{1, Z} = {0.7pt, solid},
        hline{2} = {0.5pt, dotted, gray},
    }
        Functional & Euler-Lagrange Equations & Zero Functional Solution \\
        \text{YM}~~~~\norm*{F^A}^2 & \dd_A^*F^A = 0
            & $F^A=0$ ($A$ is flat)
        \\
        \text{PYM}~~~\norm*{F^A_p}^2\, & \dd_A^*F^A_p = 0
            & $F^A=\Phi^A\om$ ($A$ is $\om$-flat)
        \\
        \text{TYM}~~ \norm*{\Phi^A\omega}^2 &
             \dd_A^*(\Phi^A\omega) = 0
            & $F^A=F^A_p$ ($A$ is trace-flat)
    \end{tblr}
    \begin{varwidth}{0.8\textwidth}
        \caption{The Yang-Mills, primitive Yang-Mills and trace Yang-Mills functionals, their Euler-Lagrange equations, and the conditions on the curvature $F^A=F^A_p + \Phi^A \om$ for the corresponding functional to vanish.
        The latter two Euler-Lagrange equations are derived in
        \cref{prop:pym:EL_equations,prop:tym:EL_equations}, 
        respectively.
        }
        \label{table:critical_points_of_functionals}
    \end{varwidth}
\end{table}

The simplest non-trivial class of critical solutions, corresponding to the functional being zero, already points to the potential interest of these functionals.  In the \gls{ym} case, the functional vanishes when $F^A=0$ and such connections are called flat connections.  The classification of flat connections for all $G$-bundles is well known to be given by the $G$-character variety of $\pi_1(M)$.  

For the \gls{pym} functional, the condition for a zero functional is $F^A_p =0$, or equivalently, in terms of the $F^A$ decomposition of \cref{FAdecomp}, 
$F^A = \Phi^A\omega\,$,
which is precisely the \defterm{symplectically flat} (or simply, \defterm{$\om$-flat}) condition introduced by Tseng and Zhou   \cite{tsengSymplecticFlatnessTwisted2022}.  Unrelated to Yang-Mills theory, the $\om$-flat condition was identified by Tseng-Zhou to be a sufficient condition that allows for twisting Tseng-Yau's elliptic complex of primitive forms \cite{tsengCohomologyHodgeTheory2012II} by a vector bundle.  Moreover, as will be described in \cref{thm:classification_of_omega_flat_connections}, $\om$-flat connections also have an algebraic classification that takes into account the symplectic structure $\om$.  

On the other hand, a connection that results in the vanishing of the \gls{tym} functional, which we will refer to as \defterm{trace-flat}, must have purely primitive curvature, $F^A = F_p^A$.  In the presence of a compatible complex structure, all traceless Hermitian Yang-Mills solutions have primitive curvature and so they represent a widely-studied class of trace-flat solutions of the \gls{tym} functional.  However, the trace-flat condition is notably a weaker condition than the Hermitian Yang-Mills condition.  Specifically, there is no requirement that $F^A$ also be of complex bidegree $(1,1)$.  In fact, the moduli space of trace-flat solutions generally is infinite-dimensional as we will show later in \cref{ex:TYM:inf_dim_family}.

\pdfbookmark[2]{Outline of the paper}{outline} %
This paper is organized as follows.
We begin in \cref{sec:preliminaries} by detailing our notations and conventions for gauge theory and reviewing the extension of the Lefschetz decomposition to the exterior covariant derivative operator, splitting it as $\dd_A = \delpa + \omega \wedge \delma$
\cite{tsengSymplecticFlatnessTwisted2022}.
Just as $\dd_A^*$ is central to the \gls{ym} equations, the $\delpa^*$ and $\delma^*$ operators also have a main role to play in describing the \gls{pym} and \gls{tym} equations, respectively.

In \cref{sec:pym:functional}, we introduce the \gls{pym} functional and describe some of its basic properties.
In particular, we derive its Euler-Lagrange equations and prove that over closed manifolds its moduli space of solutions is finite-dimensional.
We give the same treatment for the \gls{tym} functional and point out the infinite-dimensionality of its moduli space.  The equations for all three functionals are also considered together.  For example, as can be easily seen from \cref{table:critical_points_of_functionals}, since $F^A = F_p^A + \Phi^A \om$, if $A$ is a critical solution of two of the three functionals, then it must be a critical solution of all three.  Moreover, in the presence of a compatible almost complex structure, we obtain a general criterion for a critical solution of any one functional to be a critical solution of all three.
\begin{restatable}{prop}{BidegreeEquivalencies}
    \label{cor:bidegree_equivalencies}
    Let  $M^{2n}$ be a symplectic manifold equipped with a compatible triple $(\omega, g, J)$ and $A$ be a connection.  Suppose that either of the following holds:  
    \begin{enumerate}
        \item $n \geq 2$ and $ F^A \in \Omega^{1,1}(M, \ad P)$;
        
        \item $n \geq 3$ and
            $F^A_p \in \Omega^{2,0}(M, \ad P) \oplus \Omega^{0,2}(M, \ad P)$.
    \end{enumerate}
    Then,  %
   $A$ is a critical solution of any one of the three functionals -- \gls{ym}, \gls{pym}, and \gls{tym} -- if and only if  it is a critical solution of all three.
\end{restatable}

In \cref{sec:comparison_of_crit_points}, we give examples of explicit solutions on symplectic manifolds that are critical solutions of one of the functionals but not of the other two.  For instance, the 4-dimensional self-dual BPST instanton, a well-known Yang-Mills critical solution,  is not \gls{pym} with respect to the standard symplectic structure.  As we show, it is also not difficult to construct critical solutions that are \gls{pym} or \gls{tym}, but not \gls{ym}.  We further prove rigidity results on the existence and properties of critical solutions of the \gls{tym} functional and $\omega$-flat connections.
We use them to show, for example, that K3 surfaces do not admit $\omega$-flat connections on their tangent bundles with respect to any symplectic form
(\cref{ex:K3_surfaces}).
Furthermore, %
we relate the \gls{pym} functional to the cone Yang-Mills functional introduced by Tseng and Zhou in
\cite{tsengMappingConeConnections2025}.
In a special case, the zeros of the cone Yang-Mills functional coincide with the zeros of the \gls{pym} functional.
This is used to obtain the algebraic classification of $G$-bundles which admit $\omega$-flat connections. %

Finally, in \cref{sec:monotonicity_formula}, we prove a monotonicity formula as a first step towards a compactness theory of the moduli space of \gls{pym} connections.
To see the importance of such a formula, consider
Price's monotonicity formula for the \gls{ym} functional
\cite[Theorem~1.1']{priceMonotonicityFormulaYangMills1983}, which states that
\[
    f(r) = e^{ar^2}r^{4-2n}\int_{B_p(r)} \abs*{F^A}^2 \vol_g
\]
is monotone increasing whenever $A$ is \gls{ym}.
Here, $f$ is a scale-invariant $L^2$-norm of the curvature in a ball, and
$a = a_g(p)$ is some geometric constant with dimensions of inverse length squared.
Price's formula justifies $f$ as the natural quantity to detect singularities: its monotonicity implies that control of $f$ on any ball transfers to control over smaller balls, and its scale-invariance makes estimates of it geometrically meaningful.
Since singularities in \gls{ym} theory emerge as energy concentration, a monotonicity formula is a critical ingredient in the compactness theory of
\gls{ym} connections in higher dimensions
\cite{nakajimaCompactnessModuliSpace1988}.
We prove an analogous formula for \gls{pym} connections.
\begin{restatable}[Monotonicity formula]{thm}{MonotonicityFormula}
    \label{thm:monotonicity_formula}
    Let $(M^d, \omega, g)$ be a symplectic manifold with $d \geq 4$.
    Let $p \in M$. 
    Then there exist constants $r(p) = r_g(p)$ and 
    $a = a_g(p)$ such that, for any \gls{pym} connection $A$, and any
    $0 < r_1 \leq r_2 < r(p)$,
    \begin{equation}\label{eq:defect_monotonicity}
        \theta_{r_2} - \theta_{r_1}
        \geq
        4\int_{B_p(r_2) \setminus B_p(r_1)}
            e^{a\rho^2}\rho^{4-d}
            \abs{\iota_{\del_\rho}F^A_p}^2 \vol_g
        - 2\int_{r_1}^{r_2} e^{a\tau^2}\tau^{3-d} S_\tau \dd\tau\,,
    \end{equation}
    where
    \[
        \theta_r \eqdef e^{ar^2}r^{4-d}\int_{B_p(r)} \abs*{F_p^A}^2 \vol_g\,,
    \]
    and
    \[
        S_\tau \eqdef
        \int_{B_p(\tau)} \rho \innerp*{
            \iota_{\del_\rho}\omega \wedge \dd_A \Phi^A, F_p^A
        }\vol_g\,.
    \]
    If $A$ is also \gls{tym} then $S_\tau = 0$ and
    $\theta_r$ is monotone increasing.
    Moreover, the constants $r(p)$ and $a$ can be taken to agree with the corresponding constants in the \gls{ym} case.
    In particular, if $M = \R^d$ and $g$ is Euclidean, then we can take $a = 0$ and $r(p) = \infty$.
\end{restatable}
Notably, \Cref{thm:monotonicity_formula} qualitatively differs from the \gls{ym} case by the presence of the defect term $S_\tau$.
After integrating by parts, $S_\tau$ can also be written as
\[
    S_\tau = \int_M \innerp{\Phi^A \calL_X\omega, F_p^A}\vol_g\,,
\]
where $X$ is the particular compactly supported vector field used to prove the monotonicity formula (see \cref{eq:price_vector_field}).
It follows that $S_\tau$ measures how much of the primitive component of $\calL_X\omega$ is shared by $F_p^A$, weighted by $\Phi^A$.
Given that monotonicity of $\theta_r$ would be used to transfer estimates of
$\abs*{F_p^A}^2$ from larger geodesic balls to smaller geodesic balls, the sign of the defect term can affect our ability to do so.
Since the monotonicity formula is also a critical step in proving an $\epsilon$-regularity theorem that would provide a sufficient condition to conclude $L^\infty$-estimates of $\abs*{F_p^A}^2$ from $L^2$-estimates, it would also be worthwhile to examine how the defect term modifies the $\epsilon$-regularity statement.

\begin{acknowledgments}
We are sincerely grateful to Davide Parise for numerous discussions and for sharing many helpful insights related to this work.
    We would also like to thank
        Nawal Baydoun,
        Daniel Fadel,
        Joseph R. Farah,
        Connor Mooney,
        Richard Schoen,
        and 
        Jeffrey Streets
    for helpful discussions.
    The first author was partially supported by NSF grant DMS-2342135.
     The last author was partially supported by the Natural Science Foundation of Jiangxi, China (No. 20262BAC240204).
\end{acknowledgments}

\section{Preliminaries}
\label{sec:preliminaries}
In this section, we introduce our notations/conventions and review the symplectic differential operators that will be useful to study the critical solutions of the \gls{pym} and \gls{tym} functionals.

Throughout this paper, we will only concern ourselves with symplectic manifolds $(M^{d}, \om)$ of dimension $d=2n \geq 4$.  When a Riemannian metric $g$ or an almost complex structure $J$ is used, it is always assumed to be part of a compatible triple, $(\om, g, J)$.  We also assume that the orientation is compatible with the volume form $\vol_g = \omega^n/n!\,$. 
We begin with our convention for the standard \gls{ym} theory.

Let $P \to M$ be a principal bundle with compact structure group $G$.
Since $G$ is compact, there exists an Ad-invariant inner product on its Lie algebra. Together with $g$, this defines a pointwise inner product, denoted
$\innerp{\cdot, \cdot}$, on bundle-valued forms $\Omega^*(M, \ad P)$.
Integrating this over $M$ gives the $L^2$-inner product,
\[
    \norm{\eta}^2
    \eqdef \Linnerp{\eta, \eta}
    \eqdef \int_M \abs*{\eta}^2 \vol_g
    = \int_M \innerp{\eta, \eta} \vol_g\,,~~~
    \forall \eta \in \Omega^*(M, \ad P)\,.
\]
We follow the Atiyah-Bott convention for the wedge product and Lie bracket of $\ad P$-valued forms 
\cite[Section~3]{atiyahYangMillsEquationsRiemann1983}.
Namely, if
$\eta \in \Omega^k(M, \ad P)$ and $\chi \in \Omega^\ell(M, \ad P)$
are written as
\begin{align*}
    \eta = \eta^i \otimes e_i
    \quad \text{and} \quad
    \chi = \chi^j \otimes f_j\,,
\end{align*}
for real-valued forms $\eta^i, \chi^j$, and sections
$e_i, f_j \in \Gamma(\ad P)$, then their wedge product is defined as
\begin{equation}\label{eq:convention:wedge_product}
    \eta \wedge \chi
    \eqdef \eta^i \wedge \chi^j \cdot \innerp{e_i, f_j}
    \in \Omega^{k+\ell}(M)\,.
\end{equation}
Their Lie bracket is given by
\[
    \brack{\eta, \chi}
    \eqdef
    \eta^i \wedge \chi^j \otimes \brack{e_i, f_j}
    \in \Omega^{k+\ell}(M, \ad P)\,.
\]
With this convention, the $L^2$-inner product can be written as
\[
    \Linnerp{\eta, \chi}
    = \int_M \eta \wedge *\chi\,,
\]
where the Hodge star $* : \Omega^k(M, \ad P) \to \Omega^{d-k}(M, \ad P)$ extends to bundle-valued forms by acting trivially on the bundle part.
Moreover, the identity
\[
    \innerp{\brack{u, v}, w}
    = \innerp{u, \brack{v, w}}\qquad 
   \forall\, u, v, w \in \Gamma(\ad P)\,,
\]
extends to
\begin{equation}
\label{eq:bracket_wedge_identity}
    \brack{\eta, \chi} \wedge \nu = \eta \wedge \brack{\chi, \nu}\qquad
    \forall\,\eta, \chi, \nu \in \Omega^*(M, \ad P)\,.
\end{equation}

A connection 1-form $A$ on $P$ induces an exterior covariant derivative operator
\begin{align*}
    \dd_A :
        \Omega^k(M, \ad P) &\to \Omega^{k+1}(M, \ad P)
    \\* %
        \eta &\mapsto \dd\eta + \brack{A, \eta}\,.
\end{align*}
This operator has a formal $L^2$-adjoint, $\dd_A^*: \Omega^k(M, \ad P)\to \Omega^{k-1}(M, \ad P)$,  defined by
\[
    \dd_A^* = - * \dd_A * \,,
\]
where the sign is fixed since $M$ is even-dimensional.
The \gls{ym} functional is the squared $L^2$-norm of the connection's curvature $F^A$,
\begin{align}\label{YMf}
    \norm*{F^A}^2
    \eqdef \int_M \abs*{F^A}^2 \vol_g\,.
\end{align}
In order to decompose this functional, we next review the Lefschetz decomposition. We will also review the symplectic differential operators which will help us study the Euler-Lagrange equations of the resulting functionals.

\subsection{Twisted Symplectic Differentials}\label{tsympdiff}

In the presence of the non-degenerate two-form $\om$, we can apply the Lefschetz decomposition and express any differential form $\alpha_k\in \Om^k(M)$ as a polynomial in $\om\,$:  
\begin{align}\label{Lefd}
    \alpha_k
    = \sum_{p=0}^{\lfloor k/2 \rfloor} \omega^p \wedge \beta_{k - 2p}\,,
\end{align}
where $\{\beta_k, \beta_{k-2}, \ldots, \beta_{k-2\lfloor k/2 \rfloor}\}$ are primitive forms uniquely determined by $\alpha_k$ and $\om$.
Moreover, this decomposition is pointwise orthogonal with respect to
$\innerp{\cdot, \cdot}$.

The space of primitive $k$-forms is denoted by $P^k(M)$.  It is useful to introduce the Lefschetz operator $L$ and its dual operator\footnote{The dual operator $\Lambda$ can be defined symplectically without regard to any metric.  Acting on differential forms, it is just the interior product by the Poisson bivector field associated with $\om$.} which is commonly denoted by $\Lambda$:
\begin{align}\label{LLamdef}
    L :  \Omega^k(M)\, &\to \,\Omega^{k+2}(M)\,, \qquad \qquad &\Lambda :  \Omega^k(M)\, &\to \,\Omega^{k-2}(M)\,,
    \\* %
        \alpha~~~~ &\mapsto \omega\wedge\alpha\,\qquad & \quad \alpha ~~~~~&\mapsto L^*\, \alpha = (*^{-1} L \,*) \alpha\,. \nonumber
\end{align}
They allow us to characterize the space of primitive forms as follows:
\begin{equation}\label{eq:characterization_of_prim_k_forms}
    P^k(M)
    = \Omega^k(M) \cap \ker \Lambda
    = \Omega^k(M) \cap \ker L^{n-k+1}\,,
\end{equation}
where $k=0,1,\ldots, n$.  We point out that there are no primitive forms of degree $k>n$.  Also, we use the notation
\begin{align}\label{Pidef}
    \Pi :
        \Omega^k(M) &\to P^k(M)
    \\* %
        \alpha_k~~~ &\mapsto ~~\beta_k \nonumber
\end{align}
to denote the projection operator that sends forms to their primitive component.

Two useful properties that we will use are that
\[
    L^{r}\big|_{P^k(M)} : P^k(M) \to \Omega^{2r + k}(M)
\]
is injective for all $r \leq n-k$\,, and also the commutator,
\begin{align}\label{Llambdaf}
[\Lambda, L] \alpha = (n-k) \alpha\,, \qquad \forall \alpha\in\Omega^k(M)\,.
\end{align}

In \cite{tsengCohomologyHodgeTheory2012II}, Tseng and Yau observed that
\[
    d\paren*{ \om^rP^k(M) }
        \subseteq \om^rP^{k+1}(M)\oplus \omega^{r+1} P^{k-1}(M)\,.
\]
This implies the decomposition of the exterior derivative
\begin{align}\label{decomp}
d = \delp +\omega\w \delm\,,
\end{align}
where
\begin{align*}
    \delp &: \om^rP^k(M)\to \om^rP^{k+1}(M)\,,
    \\
    \delm &: \om^rP^k(M)\to \om^rP^{k-1}(M)\,,
\end{align*}
and the pair $\{\delp, \delm\}$ satisfies
\[
    \delp^2=0\,,
    \quad\delm^2=0\,,
    \quad  \om\w (\delp\delm+\delm\delp)
    =0\,.
\]
Based on this, they constructed an elliptic complex of primitive forms,
\begin{align*}
\xymatrix@R=30pt@C=30pt{
    0\; \ar[r] & \; P^{0}(M) \ar[r]^\delp &\; P^{1}(M) \ar[r]^\delp& ~ \cdots ~\ar[r]^\delp&\; P^{n-1}(M) \ar[r]^\delp&\; P^{n}(M) \ar[d]^{-\delp\delm}
    \\
    0\; & \; P^{0}(M) \ar[l]_{-\delm} &\; P^{1}(M)\ar[l]_{~~-\delm}& ~ \cdots ~\ar[l]_{~~~-\delm}&\; P^{n-1}(M) \ar[l]_{-\delm}&\; P^{n}(M)\,.
    \ar[l]_{~~~-\delm} \; 
}
\end{align*}
Since the complex is elliptic, the Laplacian operators associated to this complex are automatically elliptic.  Of particular relevance for us are the following two symplectic Laplacians\footnote{For the computation of the principal symbol of the two symplectic Laplacians, see \cite[Appendix~B]{TsengWang2022}.}:
\begin{align}\label{LaplaceP}
\Delta_+ &= \delp \delp^*+\delp^*\delp\,,  \\
\Delta_{++} & =  (\delp\delm)^*\delp\delm + (\delp \delp^*)^2\,. \label{LaplacePP}
\end{align}

The Lefschetz decomposition of differential forms $\Om^*(M)$ straightforwardly extends to a decomposition on the space of twisted differential forms $\Omega^*(M,\ad P)$.
This is because the Lefschetz decomposition only acts on the differential form component of the twisted forms. Hence, a twisted differential form $\eta_k\in \Om^k(M, \ad P)$ can also be expressed as a polynomial in $\om$:
\begin{align*}
    \eta_k
    = \sum_{p=0}^{\lfloor k/2 \rfloor} \omega^p \wedge \sigma_{k - 2p}\,,
\end{align*}
where $\{\sigma_k, \sigma_{k-2}, \ldots, \sigma_{k-2\lfloor k/2 \rfloor}\}$ are now elements of the twisted primitive forms $P^*(M, \ad P)$ and also uniquely determined by $\eta_k$ and $\om$.  The definitions of the operators $L$, $\Lambda$, and $\Pi$ in \cref{LLamdef} and \cref{Pidef} also naturally extend to act on twisted forms.  To keep the notation simple, we shall use the same symbols to denote their action on twisted forms.

As noted by Tseng and Zhou \cite[Section~2.2]{tsengSymplecticFlatnessTwisted2022}, the covariant derivative $\dd_A$, acting on the space of twisted forms $\Omega^*(M,\ad P)$, has a similar decomposition to \cref{decomp}:
\begin{align}\label{Adecomp}
\dd_A = \delpa + \om\w\delma\,,
\end{align}
where
\begin{align*}
    \delpa &: \om^rP^k(M,\ad P) \to \om^rP^{k+1}(M,\ad P)\,,
    \\
    \delma &: \om^rP^k(M,\ad P) \to \om^rP^{k-1}(M,\ad P)\,.
\end{align*}
Thus, for example, in the case of $r=0$, a primitive form $\sigma \in P^k(M, \ad P)$ is $\delma$-closed if and only if $\dd_A\sigma$ is also primitive.

In general,
$\delpa^2$, $\om\w(\delpa\delma+\delma\delpa)$, and $\delma^2$ are nontrivial and depend on the curvature, just as $\dd_A^2$ does.  The following  lemma generalizes 
\cite[(2.13)~and~(2.14)]{tsengSymplecticFlatnessTwisted2022}
for nonzero $F_p^A$.
\begin{lem}[Squares of twisted symplectic differentials]
\label{lem:squares_of_twisted_symplectic_differentials}
    Let $A$ be a connection over a symplectic manifold with
    Lefschetz decomposed curvature $F^A = F^A_p + \Phi^A\omega$.
    Let $\psi \in P^k(M, \ad P)$ be primitive.
    Then the Lefschetz decomposition
    \[
        \brack{F_p^A, \psi}
        = \Pi\brack{F_p^A, \psi}
            + \omega\wedge\sigma_{k} + \omega^{2}\wedge\sigma_{k-2}
    \]
    terminates at $\omega^2$, where $\sigma_k$ and $\sigma_{k-2}$ are primitive. 
    Furthermore,
    \begin{align*}
        \delpa^2\psi &= \Pi\brack{F_p^A, \psi}\,,
        \\
        \omega\wedge\paren*{\delpa\delma + \delma\delpa}\psi
            &= \omega\wedge\paren*{ \brack{\Phi^A, \psi} + \sigma_{k} }\,,
        \\
        \delma^2\psi
            &= \sigma_{k-2}\,.
    \end{align*}
    If $k \leq n - 1$, then
    $\paren*{\delpa\delma + \delma\delpa}\psi
    = \paren*{ \brack{\Phi^A, \psi} + \sigma_{k} }$.
\end{lem}
\begin{proof}
    Squaring the covariant derivative gives
    \begin{equation}
    \label{eq:prim_bidegree_decomp_of_dA_squared}
        \dd_A^2\psi 
        = \underbrace{\delpa^2\psi}_{P^{k+2}}
        + \underbrace{
            \omega \wedge \paren*{\delpa\delma + \delma\delpa}\psi
        }_{\omega P^{k}}
        + \underbrace{\omega^2 \wedge \delma^2\psi}_{\omega^{2}P^{k-2}}\,.
    \end{equation}
    This can also be written in terms of the curvature,
    \begin{equation}
    \label{eq:square_of_covariant_derivative}
        \dd_A^2\psi
        = \brack{F^A, \psi}
        = \brack{F^A_p, \psi}
        + \omega \wedge \brack{\Phi^A, \psi}\,.
    \end{equation}
    In particular,
    $\brack{\Phi^A, \psi} \in P^{k}(M, \ad P)$.
    Rearranging \cref{eq:square_of_covariant_derivative}
    with \cref{eq:prim_bidegree_decomp_of_dA_squared} proves that
    $\brack{F^A_p, \psi}$ terminates at $\omega^2$.
    Expand this as
    \[
        \brack{F^A_p, \psi}
        = \sigma_{k+2}
            + \omega\wedge\sigma_{k}
            + \omega^{2}\wedge\sigma_{k-2}\,.
    \]
    Comparing degrees with \cref{eq:prim_bidegree_decomp_of_dA_squared} 
    gives
    \begin{align*}
        \omega\wedge\paren*{\delpa\delma + \delma\delpa}\psi
            &= \omega\wedge\paren*{ \brack{\Phi^A, \psi} + \sigma_{k} }\,,
        \\
        \omega^2 \wedge \delma^2\psi
            &= \omega^2 \wedge \sigma_{k-2}\,.
    \end{align*}
    We can cancel out the $\omega$ from both sides of the first equality precisely when $1 \leq n - k$.
    Similarly, we can cancel out the $\omega^2$ from the second equality when $2 \leq n - (k - 2)$. If this latter condition does not hold, then $\psi = 0$ and the result follows trivially.
\end{proof}

    The operators $\delpa$ and $\delma$ have formal $L^2$-adjoints, which are denoted by $\delpa^*$ and $\delma^*$, respectively. 
A useful relation is the following:
\begin{lem}
\label{lem:agreement_of_adjoint_on_primitive_forms}
    The operators $\dd_A^*$ and $\delpa^*$ agree on primitive forms,
    \[
        \delpa^* \big|_{P^*(M,\, \ad P)}
        = \dd_A^* \big|_{P^*(M,\, \ad P)}\,.
    \]
\end{lem}
\begin{proof}
    Let $\sigma \in P^*(M, \ad P)$.
    Let $\eta \in \Omega^*(M, \ad P)$ be an arbitrary form with compact support. Then,
    \begin{align*}
        \Linnerp{\delpa^* \sigma, \eta}
        &= \Linnerp{\sigma, \delpa\eta}
        = \Linnerp{\sigma, \dd_A\eta}
            - \Linnerp{\sigma, \omega\wedge\delma\eta}\\
        &= \Linnerp{\dd_A^*\sigma, \eta}\,,
    \end{align*}
    where the very last term on the first line vanishes by the orthogonality of the Lefschetz decomposition since $\sigma$ is primitive.
\end{proof}

With the Lefschetz decomposition and its orthogonality, the \gls{ym} functional will decompose cleanly without cross terms, allowing us to study the constituent functionals independently.

\section{Primitive Yang-Mills and Trace Yang-Mills Functionals}
\label{sec:pym:functional}
Let $A$ be a connection. We express the Lefschetz decomposed curvature as
\[
    F^A = F_p^A + \Phi^A\omega\,,
\]
where explicitly,
\begin{align}\label{FAOdecomp}
    F_p^A
    = F^A - \dfrac{1}{n}\om \Lambda F^A \,,
    \quad
    \Phi^A = \dfrac{1}{n} \Lambda F^A\,.
\end{align}
We will often suppress the superscript $A$ in the curvature symbols, denoting the dependence on the connection, whenever convenient.
We call $F_p^A$ the \defterm{primitive curvature}.
Since the Lefschetz decomposition is orthogonal with respect to any compatible metric, the \gls{ym} functional splits into two components,
\begin{equation}\label{eq:lefschetz_functional_decomposition}
    \norm*{F^A}^2 = \norm*{F_p^A}^2 + \norm*{\Phi^A\omega}^2\,.
\end{equation}
The first component is what we call the \glsfmtfull{pym} functional
\begin{align}
    \norm*{F_p^A}^2
    = \int_M \abs*{F_p^A}^2 \vol_g\,.
\end{align}
The second is what we call the \glsfmtfull{tym} functional
\begin{equation}\label{eq:tym:functional}
    \norm*{\Phi^A\omega}^2
    = \int_M \abs*{\Phi^A\omega}^2 \vol_g\,.
\end{equation}
Since the Lefschetz decomposition is at the level of forms, each of these functionals independently inherits the gauge invariance of the \gls{ym} functional.

Let us point out that the \gls{tym} functional is equivalent, up to a constant factor, to the functionals
\begin{align*}
    \norm*{\Phi^A}^2 &\eqdef \int_M \abs*{\Phi^A}^2 \vol_g\,,
    \\
    \norm*{\Lambda F^A}^2 &\eqdef \int_M \abs*{\Lambda F^A}^2 \vol_g\,.
\end{align*}
By \cref{Llambdaf}, we have 
 $\abs*{\omega}^2 %
 = \innerp{1, \Lambda\omega} = n$ and also $\Lambda F^A = (\Lambda\omega)\Phi^A = n\Phi^A\,$.
Therefore,
\[
    \norm*{\Phi^A\omega}^2
    = n\norm*{\Phi^A}^2
    = \frac{1}{n}\norm*{\Lambda F^A}^2\,.
\]
This last functional can be thought of as the $\omega$-trace of the curvature, hence the name: ``trace'' Yang-Mills functional.
We present the functional in the form of \cref{eq:tym:functional} since it is most convenient for us to recognize it as the full non-primitive component of the curvature.

\begin{rmk}[Scaling properties]
    \label{rmk:rescaling_the_pym_functional}
    Under the metric rescaling $g \mapsto \lambda^2g$, the \gls{ym} functional  scales as
    \[
        \norm*{F^A}^2_{\lambda^2 g}
        = \lambda^{d-4}\norm*{F^A}^2_{g}\,.
    \]
In the present context, we have a compatible triple $(\omega, g, J)$.  To preserve compatibility,  we should instead perform the rescaling as
    $(\omega, g, J) \mapsto (\lambda^2 \om, \lambda^2 g, J)$.
Temporarily using a superscript to denote the dependence of the Lefschetz decomposition on the choice of $\omega$, it follows from \cref{FAOdecomp} that
  \[
        F_p^{\lambda^2\omega} = F_p^\omega 
        \quad\text{and}\quad
        \Phi^{\lambda^2\omega} = \lambda^{-2}\Phi^{\omega}\,,
    \]
since $\Lambda$ scales as the inverse of $\om$, or equivalently, 
    \[
        F = F_p^\omega + \Phi^\omega\omega
        = F_p^{\lambda^2\omega}
            + \Phi^{\omega}\lambda^{-2}\paren*{ \lambda^2\omega }\,.
    \]
Therefore, under rescaling by  $(\omega, g, J) \mapsto (\lambda^2 \om, \lambda^2 g, J)$, the \gls{pym} and \gls{tym} functionals scale as
    \begin{equation*}
        \norm*{F_p^{\lambda^2 \omega}}_{\lambda^2g}^2
        = \lambda^{d-4}\norm*{F_p^{\omega}}_{g}^2
        \quad\text{and}\quad
        \norm*{\Phi^{\lambda^2 \omega}(\lambda^2\omega)}_{\lambda^2g}^2
        = \lambda^{d-4} \norm*{\Phi^{\omega}\omega}_{g}^2\,,
    \end{equation*}
which are identical to the scaling of the \gls{ym} functional under just metric rescaling.    This property will be used in \cref{sec:monotonicity_formula} in the proof of the monotonicity formula for the \gls{pym} functional.
\end{rmk}

We will consider next the variation of the \gls{pym} and the \gls{tym} functionals   to derive the Euler-Lagrange equations and study the local properties of their moduli spaces of critical solutions.   
Our treatment here will mirror that of Atiyah-Bott in
\cite[Section~4]{atiyahYangMillsEquationsRiemann1983} for the \gls{ym} functional.

\subsection{Critical Solutions of the Primitive Yang-Mills Functional}
\label{sec:critical_points_of_pym}
To begin, recall that a connection $A$ is \gls{ym} if and only if
$\dd^*_A F^A = 0$. The equation for \gls{pym} connections is similar.

\needspace{2\baselineskip}
\begin{prop}[Primitive Yang-Mills equations]
\label{prop:pym:EL_equations}
    A smooth connection $A$ is a critical solution of the \gls{pym} functional if and only if it satisfies the \defterm{primitive Yang-Mills equations}:
    \[
        \delpa^* F^A_p = 0\,.
    \]
    Equivalently, $\dd_A^* F^A_p = 0$.
\end{prop}
\begin{proof}
    We parametrize the variation of the connection by
    $A_t = A + t\eta$, with $\eta \in \Omega^1(M, \ad P)$ compactly supported.
    Then,
    \begin{equation}\label{eq:curvature_on_family}
        F^{A_t} = F^A+td_A\eta+\frac{1}{2}t^2[\eta,\eta]\,.
    \end{equation}
    Projecting onto its primitive component, we have
    \begin{equation}\label{eq:primitive_curvature_on_family}
        F_p^{A_t} = F_p^A + t\delpa\eta + \frac{1}{2}t^2\Pi[\eta,\eta]\,.
    \end{equation}
    Therefore,
    \begin{align}\label{infinitesimal of F_p^2}
        \norm*{F^{A_t}_p}^2
        = \norm*{F^A_p}^2
        + 2t\Linnerp{ \delpa\eta, F^A_p }
        +t^2\left(\|\delpa\eta\|^2+ \Linnerp{ F^A_p,\Pi[\eta,\eta] } \right)
        + O(t^3)\,.
    \end{align}
    Since $\eta$ is compactly supported, it follows that
    $A$ is a critical solution if and only if
    $\delpa^* F_p^A = 0$. By \cref{lem:agreement_of_adjoint_on_primitive_forms}
    this is equivalent to $\dd_A^* F^A_p = 0$.
\end{proof}
By \Cref{prop:pym:EL_equations}, the \gls{pym} equations can be written identically to the Yang-Mills equations, with $F^A$ replaced by $F^A_p$.  

\begin{rmk}[Bianchi identity]\label{remark:primitive_laplacian}
    We can Lefschetz decompose the Bianchi identity as follows:
    \begin{align*}
    0=d_A F^A &= d_A \left(F^A_p + \Phi^A \om\right)\\
    &= \delpa F^A_p + \om \left( \delma F^A_p + d_A \Phi^A \right)\,,
    \end{align*}
    where in the second line, we applied the covariant derivative decomposition of \cref{Adecomp}, $d_A = \delpa + \om \,\delma$.  This implies, in particular,
    \begin{align}\label{eq:primitive_bianchi}
    \delpa F^A_p &= 0\,,\\  
    \delma F^A_p &= - d_A \Phi^A \,, \label{eq:omega_bianchi}
    \end{align}
    where to justify the second equation, we have noted that the $\om\w $ map on one-forms is injective when $d=2n\geq 4$.  Notice also that $\delpa F^A_p$ is a primitive three-form.  Hence, \cref{eq:primitive_bianchi} becomes a trivial condition when $d=2n=4$, since there is no $(n+1)$-degree primitive form.   But when $d= 2n>4$, 
    \cref{eq:primitive_bianchi} together with the \gls{pym} equations $\delpa^* F^A_p=0$ imply that $A$ is \gls{pym} over a closed manifold if and only if $F_p^A$ is harmonic with respect to the $\delpa$-Laplacian
        \begin{equation}\label{eq:primitive_laplacian}
            \laplace^A_+ \eqdef \delpa\delpa^* + \delpa^*\delpa\;.
        \end{equation}
    This is analogous to the statement that $A$ is \gls{ym} over a closed manifold if and only if $F$ is harmonic with respect to $\laplace^A=\dd_A \dd_A^* + \dd_A^*\dd_A$. Note that the $\delpa$-Laplacian $\laplace^A_+$ is an elliptic operator since its untwisted version $\laplace_+$ in \cref{LaplaceP} is elliptic.
\end{rmk}

The second variation of the \gls{ym} functional over a closed manifold yields a quadratic form $Q$ defined by its Hessian. We will not distinguish these two objects and simply refer to $Q$ as the Hessian of its respective functional. The Hessian also implicitly defines a Jacobi operator, or stability operator, by
\[
    \frac{1}{2} \frac{\dd^2}{\dd t^2} \norm*{F^{A + t\eta}}^2 \Big|_{t=0}
    = Q(\eta, \eta) \eqdef* \Linnerp{L^A\eta, \eta}\,.
\]
The Hessian of the \gls{ym} functional is given in
\cite[Proposition~4.10]{atiyahYangMillsEquationsRiemann1983} and has the form
\begin{equation}
\label{eq:hessian_ym}
    Q(\eta, \eta) \eqdef \Linnerp{ \dd_A^*d_A\eta+*[*F,\eta],\eta }\,.
\end{equation}
The Hessian for the \gls{pym} functional is similar.
\begin{prop}[Hessian of the primitive Yang-Mills functional]
\label{prop:hessian_of_pym_functional}
    Let $M$ be closed.
    The Hessian of the \gls{pym} functional is given by the quadratic form
    \[
        Q_p(\eta,\eta)
        = \Linnerp{ L_p^A \eta, \eta }\,,
    \]
    where $\eta \in \Omega^1(M, \ad P)$ and
    $L_p^A$ is the corresponding Jacobi operator
    \[
        L_p^A\eta \eqdef \delpa^*\delpa\eta+*[*F^A_p,\eta]\,.
    \]
\end{prop}
\begin{proof}
    By \cref{infinitesimal of F_p^2},
    $$
        Q_p(\eta,\eta)=
          \frac{1}{2} \frac{\dd^2}{\dd t^2} \norm*{F^{A + t\eta}_p}^2 \Big|_{t=0}=\|\delpa\eta\|^2 + \Linnerp{ F_p^A,\Pi[\eta,\eta] }\,.
    $$
    As $\delpa\eta$ is primitive, it is orthogonal to $\omega\, \delma\eta$.
    Since $M$ is closed,
    \[
        \norm*{\delpa\eta}^2
        = \Linnerp{\delpa\eta,d_A\eta}
        = \Linnerp{\dd_A^*\delpa\eta,\eta}\,.
    \]
    For the remaining term,
    \begin{align*}
        \Linnerp{ \Pi[\eta,\eta], F_p^A }
        & %
        = \Linnerp{ [\eta,\eta], F_p^A}
           & \paren*{ \text{Lefschetz orthogonality} }
        \\* &
        = \int_M [\eta,\eta]\wedge *F_p^A
        \\ &
        = \int_M \eta\wedge[\eta,*F_p^A]
            & \text{\Cref{eq:bracket_wedge_identity}}
        \\ &
        = -\int_M \eta\wedge **^{-1}[*F_p^A,\eta]
        \\ &
        = \int_M \eta\wedge **[*F_p^A,\eta]
           & \paren*{ \text{$*^{-1} = -*$ on odd-degree forms} }
        \\ &
        = \Linnerp{ \eta, *[*F_p^A,\eta] }
    \,.\end{align*}
    Therefore,
    \[
        L_p^A \eta = \dd_A^* \delpa\eta + *[*F^A_p,\eta]\,.
    \]
    The result then follows from
    \cref{lem:agreement_of_adjoint_on_primitive_forms}.
\end{proof}

At a critical point, the Jacobi operator describes the formal tangent space of \gls{ym} connections by its kernel. In particular, if $A_t$ is a path of \gls{ym} connections starting at $A$, then for
$\eta = \ddt A_t \big|_{t=0} \in \Omega^1(M, \ad P)$, it must be that
$\eta \in \ker L^A$.
The corresponding statement for the \gls{pym} functional is also true.
\begin{thm}[Formal tangent space of \glsfmtlong{pym} connections]
\label{thm:formal_tangent_space_of_pym_connections}
    Let $M$ be closed.
    If $\eta \in \Omega^1(M, \ad P)$ is a tangent vector at $A$ to the space of \gls{pym} connections, then 
    \[
        \eta \in \ker L_p^A\,.
    \]
    In particular, the space of \gls{pym} connections modulo gauge equivalence is finite-dimensional.
\end{thm}
\begin{proof}
    Let $A_t=A+t\eta+O(t^2)$. Since
    $\ddt (\dd_{A_t}) \big|_{t=0} = [\eta, \cdot]$ and
    $\ddt F^{A_t}_p \big|_{t=0} = \delpa\eta$, we have
    \[
        \ddt \paren*{ *\,\dd_{A_t}*F^{A_t}_p } \Big|_{t=0}
        = *\,\dd_A*(\delpa\eta) + *[\eta,*F^A_p]
        = -L^A_p\eta\,.
    \]
    Thus, if $A_t$ is a family of \gls{pym} connections, then $\eta \in \ker L_p^A$.
    
    For the second statement, note the connections $A_t$ are all gauge equivalent to $A$ if and only if $\eta$ is $\dd_A$-exact.
    Since $\dd_A = \delpa$ when acting on zero-forms, the complement of the $\delpa$-exact one-forms is precisely the space of $\dd_A^*$-closed forms, which by \cref{lem:agreement_of_adjoint_on_primitive_forms} is equivalent to
    $\delpa^*$-closed forms. Thus, the formal tangent space of \gls{pym} connections up to gauge equivalence can be identified with the space of
    $\eta\in\Omega^1(M,\ad P)$ satisfying $L^A_p\eta=0$ and $\delpa^*\eta=0$.
    Equivalently,
    \[
        \delpa^*\eta = 0\,,
        \qquad\text{and}\qquad
        \laplace_+^A \eta + *[*F_p^A,\eta] = 0\,,
    \]
    where $\laplace_+^A$ is the $\delpa$-Laplacian in \cref{eq:primitive_laplacian}.
    Since $\laplace_+^A$ is elliptic, it follows that the moduli space of
    \gls{pym} connections is finite-dimensional.
\end{proof}

\subsection{Critical Solutions of the Trace Yang-Mills Functional}
\label{sec:the_three_functionals}
Consider now the \gls{tym} functional $\norm{\Phi^A\omega}^2$.
By arguments similar to those for the \gls{pym} functional, we obtain the following.
\begin{prop}[Euler-Lagrange equations for the \glsfmtlong{tym} functional]
\label{prop:tym:EL_equations}
    Let $A$ be a connection over $M$. The following are equivalent:
    \begin{center}
        \begin{tblr}{
            colspec = {ll},
            row{1-Z} = {mode = math}
        }
            \mathrm{(i)}\ A \text{ is a critical solution of the \gls{tym} functional}\,;\qquad
            &
            \mathrm{(iv)}\ \dd_A \Phi^A = 0\,;
            \\
            \mathrm{(ii)}\ \dd_A^*(\Phi^A \omega) = 0\,;
            &
            \mathrm{(v)}\ \delpa\Phi^A = 0\,;
            \\
            \mathrm{(iii)}\ \delma^*\Phi^A = 0\,;
            &
            \mathrm{(vi)}\ \dd_A F_p^A = 0\,.
        \end{tblr}
    \end{center}
\end{prop}
Before proving this, it will be useful to expand $\Phi^{A_t}\omega$ in terms of $t$.
Let $\eta \in \Omega^1(M, \ad P)$ be compactly supported.
Set $A_t = A + t\eta$.
Then, using \cref{eq:curvature_on_family,eq:primitive_curvature_on_family}, we have 
\begin{equation}\label{eq:trace_curvature_to_higher_order}
    \Phi^{A_t}\omega = F^{A_t} - F_p^{A_t}
    = \Phi^A\omega + t\,\omega\,\delma \eta
        + \frac{1}{2}t^2\paren*{\brack{\eta, \eta} - \Pi\brack{\eta, \eta}}\,.
\end{equation}
The order $t^2$ term is precisely the non-primitive component of the two-form $\brack{\eta, \eta}$.
Similar to \cref{FAOdecomp}, this non-primitive component can be expressed concisely as 
\[
    \brack{\eta, \eta} - \Pi\brack{\eta, \eta}
    = \frac{1}{n} \omega \Lambda \brack{\eta, \eta}\,.
\]
Hence, we find
\begin{align}\label{eq:Phinormsq}
    \norm*{\Phi^{A_t}\om}^2
    = \norm*{\Phi^A\om}^2 + & 2t\Linnerp{\Phi^A\om, \om\, \delma\eta}
    \\*\nonumber &
    + t^2\left(\|\om\,\delma\eta\|^2+ \Linnerp{ \Phi^A\om, \frac{1}{n}\omega \Lambda[\eta,\eta] } \right)
            + O(t^3)\,.
\end{align}
From this we can read off the Euler-Lagrange equations and the Hessian.
\begin{proof}[Proof of \Cref{prop:tym:EL_equations}]
    (i) $\Longleftrightarrow$ (ii):
From \cref{eq:Phinormsq}, we have
    \begin{equation}\label{eq:first_variation_tym}
        \frac{1}{2} \frac{\dd}{\dd t} \norm*{\Phi^{A + t\eta}\omega}^2 \Big|_{t=0}
        = \Linnerp{\Phi^A\omega, \omega\, \delma \eta}\,.
    \end{equation}
    Lefschetz orthogonality implies
    \[
        \Linnerp{\Phi^A\omega, \omega\, \delma \eta}
        = \Linnerp{\Phi^A\omega, \dd_A \eta}
        = \Linnerp{\dd^*_A (\Phi^A\omega), \eta}\,.
    \]

    (i) $\Longleftrightarrow$ (iii): Alternatively, from
    \Cref{eq:first_variation_tym} and using the fact that $\Lambda\omega \Phi^A= n \Phi^A$, we obtain
    \[
        \frac{1}{2} \frac{\dd}{\dd t} \norm*{\Phi^{A + t\eta}\omega}^2 \Big|_{t=0}
        = \Linnerp{\Phi^A\omega, \omega\, \delma \eta}
        = n\Linnerp{\Phi, \delma \eta}
        = n\Linnerp{\delma^* \Phi, \eta}\,.
    \]

    (ii) $\Longleftrightarrow$ (iv):
    Applying $d_A^* = - * d_A * $ and the relation $*\, \om = \om^{n-1}/(n-1)!$, which holds since $g$ is compatible with $\omega$,  we have
    \[
        \dd_A^*(\Phi^A\omega)
        =-\frac{1}{(n-1)!}*d_A(\omega^{n-1}\Phi^A)
        =-\frac{1}{(n-1)!}* \paren*{ \omega^{n-1}  d_A\Phi^A }\,.
    \]
    So $A$ is a critical solution if and only if
    $\omega^{n-1} d_A\Phi^A=0$.
By the injectivity of the map $\omega^{n-1} : \Omega^1(M)\to\Omega^{2n-1}(M)$, this is equivalent to $\dd_A\Phi^A = 0$.

    (iv) $\Longleftrightarrow$ (v):
    This is just from expanding $\dd_A = \delpa + \omega \wedge \delma$ and recognizing that, since $\delma$ lowers the degree of a form by one, $\Phi^A$ is trivially $\delma$-closed.

    (iv) $\Longleftrightarrow$ (vi):
    This equivalence follows immediately from the Bianchi identity conditions in \cref{eq:primitive_bianchi,eq:omega_bianchi}, i.e.,  
    \[
        \dd_A F^A = \dd_A F_p^A + \omega \, \dd_A \Phi^A = 0\,.
        \qedhere
    \]
\end{proof}

Combining the Euler-Lagrange equations for each functional (\cref{prop:pym:EL_equations,prop:tym:EL_equations}, see also \cref{table:critical_points_of_functionals}), we obtain the following useful criterion for a connection to be a critical solution of all three functionals.
\begin{cor}[2 out of 3 property]
\label{lem:2_out_of_3}
    If $A$ is a critical solution of two of the three functionals, 
    $\norm{F^A}^2,\norm{F_p^A}^2$, and $\norm{\Phi^A\omega}^2$, then it is a critical solution of all three.
\end{cor}

\begin{ex}[Zeros of the functionals]
    \label{ex:zeros}
    Flat connections $F^A = 0$ are zeros of all three functionals.
    In the same spirit, $\omega$-flat connections are the zeros of the
    \gls{pym} functional, and so they are \gls{pym}.
    Their curvature is given by $F^A = \Phi^A\omega$, and the Bianchi identity implies that they are \gls{tym}. Hence, $\omega$-flat connections are also critical solutions of all three functionals.
    The zeros of the \gls{tym} functional, trace-flat connections, have curvature that is entirely primitive, $F^A = F_p^A$.
    These connections are not necessarily \gls{ym} or \gls{pym} in general as can be seen in the family of trace-flat solutions in \cref{ex:closed_kahler_nPYM_not_YM}.    
\end{ex}

Of note, \cref{prop:tym:EL_equations} shows that when $A$ is \gls{tym},
then $\Phi^A$ is covariantly constant.
This places a strong constraint on $\Phi^A$. 
In particular, since $\Phi^A \in \Omega^0(M, \ad P)$,
$\nabla^A\Phi^A = \dd_A \Phi^A = 0$.
Thus, $A$ is reducible whenever $\Phi^A$ is nonzero and $G$ is semisimple.

For the second variation, we have the following.
\begin{prop}[Hessian of the \glsfmtlong{tym} functional]
    Let $M$ be closed.
    The Hessian of the \gls{tym} functional is
    \[
        Q_{\omega}(\eta,\eta)
        = \Linnerp{ L_\omega^A\eta, \eta }\,,
    \]
    where $\eta \in \Omega^1(M, \ad P)$ and
    \[
        L_\omega^A \eta
        \eqdef \dd_A^*(\omega \,\delma\eta) + *[*(\Phi^A\omega),\eta]
    \]
    is the corresponding Jacobi operator.
    In particular, if $\eta$ is a tangent vector at $A$ to the space of \gls{tym} connections, then $L^A_\omega\eta = 0$.
\end{prop}
\begin{proof}
The proof is analogous to the proofs of \cref{prop:hessian_of_pym_functional,thm:formal_tangent_space_of_pym_connections}.
    By \cref{eq:trace_curvature_to_higher_order},
\[Q_\omega(\eta, \eta)
=\frac{1}{2} \frac{\dd^2}{\dd t^2} \norm*{\Phi^{A + t\eta}\om}^2 \Big|_{t=0}
= \norm*{\omega\,\delma \eta}^2
+ \Linnerp{\frac{1}{n}\omega \Lambda[\eta,\eta] , \Phi^A\omega}\,.
\]
    Using Lefschetz orthogonality, this first term is
    \[
        \norm*{\omega\,\delma \eta}^2
        = \Linnerp{\omega\,\delma \eta, \dd_A \eta}
        = \Linnerp{\dd_A^*(\omega\,\delma \eta), \eta}\,.
    \]
    The remaining term is
    \[
        \Linnerp{ \frac{1}{n}\omega \Lambda[\eta,\eta] , \Phi^A\omega }
        = \Linnerp{ [\eta,\eta], \Phi^A\omega}
        = \Linnerp{ \eta, *[*(\Phi^A\omega),\eta]}\,.
    \]
    For the last claim, let $A_t = A + t\eta + O(t^2)$, and let $L^A$ be the Jacobi operator associated to the \gls{ym} functional in
    \cref{eq:hessian_ym}.
    Then
    \[
        \ddt \dd_{A_t}^*\Phi^{A_t}\omega
        = \ddt \dd_{A_t}^*\paren*{ F^{A_t} - F_p^{A_t} }
        = L^A\eta - L^A_p\eta = L^A_\omega\eta\,.
    \]
    Thus, if $A_t$ is a family of \gls{tym} connections, then $\eta \in \ker L_\omega^A$.
\end{proof}

\begin{rmk}[Nonellipticity of the \glsfmtlong{tym} Jacobi operator]
\label{rmk:nonellipticity_of_tym_jacobi_operator}
    Unlike that of the \gls{pym} functional, the Jacobi operator of the \gls{tym} functional is not elliptic even after a gauge-fixing condition.
    To see this, we compute its principal symbol:
    \begin{align*}
        \sigma_2(L^A_\omega)(\xi)\eta
        & %
        = \sigma_1(\dd_A^*) \circ
            \sigma_1\paren*{ \omega\delma }(\xi)\eta
        \\ &
        = \sigma_1(\dd_A^*) \circ
            \sigma_1\paren*{ \frac{1}{n} L\Lambda \dd_A }(\xi)\eta
        \\ &
        = \sigma_1(\dd_A^*) \circ \frac{1}{n} L\Lambda
            \circ \sigma_1(\dd_A)(\xi)\eta
        \\ &
        = -\frac{1}{n}\Lambda \paren*{ \xi \wedge \eta }
            \iota_{\xi^\sharp}\omega
    \,.\end{align*}
    As a map,
    $\sigma_2(L^A_\omega)(\xi)
    : T^*_xM \otimes \mathfrak{g} \to T^*_xM \otimes \mathfrak{g}$
    has rank $\dim G$ on a $(2n \dim G)$-dimensional space, and hence fails to be an isomorphism.
    A gauge-fixing condition does not repair this either, and so $L^A_\omega$ is not elliptic.
    Thus, the formal tangent space is not given by the kernel of an elliptic operator.
\end{rmk}

It is not difficult to see that the moduli space of \gls{tym} connections is generically infinite-dimensional.  We demonstrate this with a simple example on $T^4$.

\begin{ex}[Infinite-dimensional family of trace-flat solutions of the \gls{tym} functional]\label{ex:TYM:inf_dim_family}
Consider the trivial $U(1)$-bundle over $T^4$ with coordinates $(x_i, y)$ where $x_i$, for $i=1,2,3,4$, are $\R^4$ coordinates with identifications $x_i\sim x_i + 2\pi$, and $y\sim y+2\pi$ is the coordinate on the circle fiber. 

It is not difficult to construct trace-flat ($\Phi^A=0$) solutions of the \gls{tym} functional.  For instance, let $u \in C^\infty(T^2)$ be a function of only $x_1$ and $x_2$.  Consider the following $U(1)$-connection with corresponding curvature:
\begin{align}\label{FAuu}
        A_{u} = \dd y + u(x_1, x_2)\dx_4\,,
    \qquad
        F^{A_{u}}
        = \dd u \wedge \dx_4\,.
    \end{align}
Clearly, $F^{A_u}$ is primitive for any choice of $u$, and hence, it represents a trace-flat solution of the \gls{tym} functional.  Now, recall that any two $U(1)$-connections, $A$ and $A'$, on the same bundle are gauge-equivalent if and only if (1) $\alpha = A - A'$ is closed, and (2) $\int_{\gamma}\alpha \in 2\pi\Z$ for every loop $\gamma$.
The first condition implies $u - u'$ is constant.  The second condition implies $u - u'$ is an integer.  Altogether, modding out by gauge equivalence, we have obtained an infinite-dimensional family of trace-flat solutions parameterized by
\begin{align*}
        \set*{A_{u} : u \in C^\infty(T^2)}/{\sim}
       \quad \cong \quad C^\infty(T^2)/\Z ~.
\end{align*}
\end{ex}

\subsection{Comparison of the Euler-Lagrange Equations}
\label{sec:EL_equations_revisited}
In this subsection, we will express the Euler-Lagrange equations for the \gls{ym}, \gls{pym}, and \gls{tym} functionals in terms of the twisted symplectic differentials reviewed in \cref{tsympdiff}.  Specifically, we will make use of the almost complex structure $J$ of the compatible triple $(\om, g, J)$ to express the equations in terms of $\delma$.

In the presence of an almost complex structure, differential forms have a bidegree decomposition:
\[
    \Omega^k(M, \mathbb{C}) = \bigoplus_{k = p + q} \Omega^{p, q}(M)\,.
\]
With $(\omega, g, J)$ being a compatible triple, $\omega$ has bidegree $(1, 1)$.  Applying both the complex bidegree decomposition and Lefschetz decomposition to the curvature, we have  
\[
    F^A = \underbrace{F^{2,0} + F^{0,2} + F_p^{1,1}}_{F_p^A} + \Phi^A\omega\,.
\]
Let us denote by 
\begin{align*}
    \Pi^{p,q} :
        \Omega^{*}(M, \ad P) &\to \Omega^{p,q}(M, \ad P)
\end{align*}
the projection operator that maps forms onto their $(p, q)$-component.  We also recall the $\calJ$ operator which acts by multiplication by $i^{p-q}$ on $(p,q)$-components:
\begin{align*}
    \calJ :
        \Omega^k(M, \ad P) &\to \Omega^k(M, \ad P)
    \\
        \alpha &\mapsto \sum_{p+q=k} i^{p-q}\Pi^{p,q}(\alpha)\,.
\end{align*}
In particular,
\[
    \calJ(F_p) = -\paren[\big]{ \underbrace{F^{2, 0} + F^{0, 2}}_{F^{\calJ-}} }
        + F_p^{1,1}\,,
\]
where we have introduced the notation $F^{\calJ-}$ to denote the components of the primitive curvature with eigenvalue $-1$ with respect to $\calJ$ (not to be confused with the \gls{asd} components of the curvature over a 4-manifold). %

As we have seen, the Euler-Lagrange equations for $\norm{F^A}^2$, $\norm{F_p^A}^2$, and $\norm{\Phi^A\omega}^2$ can all be written in terms of $\dd_A^* = - * \dd_A *$.
To understand these expressions further, it will be useful to have a way to explicitly compute the Hodge star of any given form.
The Weil relation (see, for instance, 
\cite[Proposition~1.2.31]{huybrechtsComplexGeometryIntroduction2005})
does exactly this.
In particular, for any primitive $k$-form $\beta \in P^k(M)$,
\begin{equation}
    \label{eq:weil_relation}
    *\,L^r \beta
    = (-1)^{\frac{ k(k+1) }{ 2 } }\frac{ r! }{ (n-k-r)! } 
        L^{n-k-r}\calJ(\beta)\,.
\end{equation}
The Weil relation will allow us to remove the Hodge star and express the Euler-Lagrange equations purely %
in terms of $\delma$.  
\begin{prop}[Equivalent conditions for \gls{ym}, \gls{pym}, and \gls{tym}]
\label{prop:kahler:yang_mills:equations}
    Let $A$ be a connection over a symplectic manifold $(M^{2n}, \omega, g, J)$ with $n \geq 2$.
    Decompose its curvature as
    \[
        F^A = \paren*{ F^{\calJ-} + F_p^{1,1} } + \Phi^A\omega\,,
        \text{ where }
        F^{\calJ-} = F^{2,0} + F^{0,2}\,.
    \]
    We have the following equivalent conditions as given in \cref{table:updated_table_for_critical_points} for critical solutions of the \gls{ym}, \gls{pym}, and \gls{tym} functionals.
\end{prop}
\begin{table}[t]
    \centering
    \begin{tblr}{
        columns = {halign = c},
        row{1} = {font=\bfseries},
        row{2-Z} = {mode = math},
        hline{1, Z} = {0.7pt, solid},
        hline{2} = {0.5pt, dotted, gray},
    }
        Functional & Equation for Critical Solutions \\
        \norm*{F^A}^2
            & \delma\paren*{ F_p^{1,1} - \frac{ n-2 }{ n } F^{\calJ-} } = 0
        \\
        \norm*{F^A_p}^2 &
            \delma\paren*{ F_p^{1,1} - F^{\calJ-} } = 0
        \\
        \norm*{\Phi^A\omega}^2 
            & \delma\paren*{ F_p^{1,1} + F^{\calJ-} } = 0
    \end{tblr}
    \begin{varwidth}{0.8\textwidth} 
        \caption{The critical solutions of the various functionals encoded
        entirely in the primitive curvature and the $\delma$-operator. The bidegree decomposition comes from the compatible triple $(\omega, g, J)$. Here, $F^{\calJ-} = F^{2,0} + F^{0,2}$.
        }
        \label{table:updated_table_for_critical_points}
    \end{varwidth}
\end{table}

\begin{proof}
For the \gls{ym} equations, we find from applying the Weil relation,    
    \begin{align*}
        0=\dd_A * F^A
        &= \dd_A * F_p^A + \dd_A * L\Phi^A
        \\ &
        =  -\frac{ 1 }{ (n-2)! } L^{n-2}\dd_A
            \calJ(F_p^A)
        +  \frac{ 1 }{ (n-1)! } L^{n-1}\dd_A \Phi^A\\
        &=-\frac{ 1 }{ (n-2)! } L^{n-1}\delma
            \calJ(F_p^A)+  \frac{ 1 }{ (n-1)! } L^{n-1}\dd_A \Phi^A\,,
    \end{align*}
    where in the third line, we applied the decomposition, $\dd_A = \delpa + L \delma$, and noted that  $\delpa \calJ(F_p^A)$ being a primitive three-form sits in the kernel of $L^{n-2}$.  Furthermore, since $L^{n-1}$ acts injectively on one-forms, this implies 
\begin{align*}
    0&= (n-1)\delma \calJ(F_p^A)-\dd_A \Phi^A \\*
     &=  (n-1)\delma \calJ(F_p^A) +\delma F_p^A \\
     &=  (n-1)\delma (F^{1,1}_p- F^{\calJ-}) + \delma (F^{1,1}_p+ F^{\calJ-})\\
     & = n\, \delma\paren*{ F_p^{1,1} - \frac{ n-2 }{ n } F^{\calJ-} }\,,
\end{align*}
where in the second line, we used the Bianchi identity relation \cref{eq:omega_bianchi}, $d_A \Phi^A =-\delma F^A_p $.

By similar arguments for the \gls{pym} equations, we find 
\begin{align*}
0= \dd_A * F_p^A &= -\frac{ 1 }{ (n-2)! }L^{n-2} \dd_A \calJ(F_p^A)\\
&=-\frac{ 1 }{ (n-2)! }L^{n-1}\delma\left(F^{1,1}_p- F^{\calJ-}\right)\,.
\end{align*}
The desired condition then follows from the injectivity of the $L^{n-1}$ map on one-forms.

Lastly, for the \gls{tym} equations, we start with the form given in \cref{prop:tym:EL_equations}(vi):
\begin{align*}
0=\dd_A F_p^A = \delpa F_p^A + \om \,\delma F_p^A\,.
\end{align*}
The first term vanishes by the Bianchi identity \cref{eq:primitive_bianchi} and replacing $F_p^A= F^{1,1}_p +F^{\calJ-}$ gives the desired result.
\end{proof}

As a check of the critical solution conditions, \cref{prop:kahler:yang_mills:equations} or \cref{table:updated_table_for_critical_points} give another proof of \cref{lem:2_out_of_3}.
\Cref{table:updated_table_for_critical_points}
also implies the following.
\begin{cor}
    Both $F_p^{1,1}$ and $F^{\calJ-}$ are $\delma$-closed if and only if $A$ is a critical solution of all three functionals.
\end{cor}
It also gives a proof of \cref{cor:bidegree_equivalencies}.

\BidegreeEquivalencies*

\begin{ex}
    A Hermitian \gls{ym} connection $A$ has curvature
    \[
        F^A = F_p^{1,1} + \lambda I\omega\,,
    \]
    where $\lambda$ is a constant and $I$ is the identity map.
    Thus, $\Phi^A = \lambda I$, which is $\dd_A$-closed and hence \gls{tym}.
    Since $F^A$ has bidegree $(1,1)$, \cref{cor:bidegree_equivalencies} gives another proof that Hermitian \gls{ym} connections over Kähler manifolds are also \gls{ym} and \gls{pym}.
\end{ex}

\section{Examples and Properties of Critical Solutions}
\label{sec:comparison_of_crit_points}
In this section, we give explicit examples of the critical solutions of the three functionals and compare their properties. To facilitate the discussion, we introduce the notation
$\critYM(P)$, $\critPYM(P)$, and $\critTYM(P)$ for the spaces of critical solutions of 
$\norm{F^A}^2$, $\norm{F_p^A}^2$, and $\norm{\Phi^A\omega}^2$, respectively.

Before presenting the examples, let us first give a sufficient condition for
$\critYM = \critPYM$ in the case when $G$ is abelian.
\begin{prop}\label{cor:Kahler:ym_iff_pym_for_G_abelian}
    A principal bundle with abelian structure group
    over a closed Kähler manifold has $\critYM = \critPYM$ and
    $\critYM \subseteq \critTYM$.
\end{prop}
\begin{proof}
    Suppose first that $A$ is \gls{ym}. Then, by the Kähler identities,
    \[
        \laplace\Phi^A
        = \dfrac{1}{n}\laplace\Lambda F^A
        = \dfrac{1}{n}\Lambda\laplace F^A = 0\,.
    \]
    Since $M$ is closed, $\dd \Phi^A = 0$, which implies $A$ is \gls{tym}.
    Applying the \nameref{lem:2_out_of_3} (\cref{lem:2_out_of_3}), $A$ must also be \gls{pym}.

    Now suppose that $A$ is \gls{pym}. Then $F_p^A$ is $\delp$-harmonic.
    Moreover, since $F^A$ is $\dd$-closed by the Bianchi identity, Hodge theory guarantees that there exists a $\dd$-harmonic representative of the cohomology class $[F^A]$, which we can denote by ${\tilde F}^{\tilde A} = F^A + \dd\xi = \dd(A + \xi)$.
    Since $M$ is closed, this implies that ${\tilde A} = A + \xi$ is \gls{ym}.
    Applying the first part of the proof to $\tilde F^{\tilde A}$, we have that
    ${\tilde F}_p^{\tilde A} = \Pi(F^A + d\xi) = F_p^A + \delp\xi$ is also $\delp$-harmonic.
    What remains to show is $\dd\xi = 0$ to prove the claim.

    Since both $F_p^A$ and $\tilde{F}_p^{\tilde A}$ are $\delp^*$-closed,
    $\delp \xi = {\tilde F}_p^{\tilde A} - F_p^A$ is also $\delp^*$-closed. Taking the $L^2$-norm,
    \[
        \norm*{\delp \xi}^2
        = \Linnerp{\delp \xi, \delp \xi}
        = \Linnerp{\xi, \delp^* \delp \xi}
        = 0\,,
    \]
    where again, we have used the fact that $M$ is closed.
    Therefore, $\delp\xi = 0$.

    This implies that $d\xi = \delp\xi + \omega \wedge \delm\xi = f\omega$ for some function $f \in \Omega^0(M)$. Moreover, $ \dd^2 \xi = \omega \w \dd f  = 0$. With the Lefschetz operator being injective on 1-forms, we conclude that $f$ must be a constant. But $[\omega] \in H^2(M)$ is nontrivial for closed symplectic manifolds, so $f$ must actually be zero. Thus, $\dd\xi = 0$, implying $F^A = \tilde{F}^{\tilde A}$ is harmonic and $A$ is \gls{ym}.
\end{proof}
Below, we will give examples which make clear that $\critYM \neq \critPYM$ in general.  We will also show in \cref{ex:closed_kahler_nPYM_not_YM} that the containment $\critYM \subseteq \critTYM$ can be proper.

\subsection{Detailed Examples}
\label{sec:detailed_examples}
We first show that the closed assumption in
\cref{cor:Kahler:ym_iff_pym_for_G_abelian} cannot be dropped.
\begin{ex}[A \gls{pym} solution that is neither \gls{ym} nor \gls{tym}]\label{ex:closed_kahler_cannot_be_relaxed}
    Consider $M = \R^4$ as a K\"ahler manifold equipped with the standard Euclidean metric and the standard symplectic form 
    $\omega = \dx_1 \wedge \dx_2 + \dx_3 \wedge \dx_4$.
    Take the $U(1)$ connection form to be
    \[
        A = \dd y + \frac{1}{2} x_1^2 \dx_2 + x_1 x_3 \dx_4\,.
    \]
    Its curvature is
    \[
        F^A = \underbrace{
            x_1 \dx_1 \wedge \dx_2 + x_1 \dx_3 \wedge \dx_4
        }_{x_1 \omega}
        + \underbrace{
            x_3 \dx_1 \wedge \dx_4
        }_{F_p}\,.
    \]
    By direct computation, we can check that $A$ is \gls{pym}:
    \[
        \dd * F_p^A = 
        \dd \paren*{ x_3 \dx_2 \wedge \dx_3 }
        = 0\,.
    \]
    However, $\dd \Phi^A = \dd x_1 \neq 0$,
    so $A$ is not \gls{tym}. By the \nameref{lem:2_out_of_3}, $A$ also cannot be \gls{ym}. 
\end{ex}

The following example also on a $U(1)$ bundle shows that a \gls{tym} connection need not be \gls{ym}.
\begin{ex}[A \gls{tym} solution that is neither \gls{ym} nor \gls{pym}]\label{ex:closed_kahler_nPYM_not_YM}
We recall the family of trace-flat solutions given in \cref{ex:TYM:inf_dim_family} on a $U(1)$-bundle over $T^4 = \R^4/2\pi\Z^4$ equipped with the standard metric and
    symplectic form inherited from $\R^4$.
    
     To see that this family is not \gls{ym} in general, we compute from the expression of $F^{A_u}$ in \eqref{FAuu}: 
    \[
        \dd * F^{A_u}
        = \dd *\paren*{
            u_{x_1} \dx_1 \wedge \dx_4
            + u_{x_2} \dx_2 \wedge \dx_4
        }
        = \paren*{
            u_{x_1x_1}
            + u_{x_2x_2}
        }
        \dx_1 \wedge \dx_2 \wedge \dx_3\,.
    \]
    This vanishes if and only if $u$ is harmonic, and hence constant since $T^2$ is closed.
    Since $F_p^{A_u}=F^{A_u}$, it follows that the family is not \gls{pym} in general either.
\end{ex}

We next consider a $U(1)$-bundle over a non-Kähler closed symplectic 4-manifold.
We write down a \gls{ym} connection that is not \gls{pym}.
In fact, by the existence of such a connection, and appealing to
\cref{cor:Kahler:ym_iff_pym_for_G_abelian}, we can conclude that the manifold is not Kähler without directly computing the Nijenhuis tensor.
We then use Hodge theory to prove the existence of a connection that is
\gls{pym} but not \gls{ym}. 
\begin{ex}[A \gls{ym} solution that is neither \gls{pym} nor \gls{tym}]\label{ex:T4_ym_not_pym_and_vice_versa}
    Consider the 4-torus $T^4$ again equipped with the standard symplectic form, but now take the Riemannian metric to be 
    $$
        g = dx_1^2+dx_2^2+\frac{1}{f}dx_3^2+fdx_4^2\,,
        \quad\text{where}\quad
        f = \dfrac{3+2\sin 2x_2\cos x_3}{1-\frac{1}{2}\sin 2x_2\cos x_3} > 0\,.
    $$
    Notice that $g$ here is compatible with $\omega$.

    Construct a circle bundle $X=\R^5/{\sim}$ over $T^4$ by identifying
    \begin{align*}
        x_i\sim x_i+2\pi n_i
        \quad\text{and}\quad
        y\sim y+2\pi n_5 - n_1x_3\,,
        \quad\text{for $i=1, \ldots, 4$, and $n_i, n_5 \in \Z$}\,.
    \end{align*}
    Let the connection be
    \[
        A = dy + \frac{1}{2\pi}x_1 dx_3 + \frac{1}{4\pi}\sin 2x_2 \sin x_3 dx_1\,.
    \]
    Its curvature
    \[
        F^A = \dd A
        = -\frac{1}{2\pi}\cos 2x_2\sin x_3 \dx_1\wedge \dx_2
        + \frac{1}{2\pi} \left(1-\frac{1}{2}\sin 2x_2\cos x_3 \right)
            dx_1\wedge dx_3
    \]
    is a two-form on $T^4$.
    Moreover,
    \begin{align*}
        d*F^A
        &= \frac{1}{2\pi} \dd \paren[\bigg]{
            -\cos 2x_2\sin x_3 dx_3 \wedge dx_4
            - \left(1-\frac{1}{2}\sin 2x_2\cos x_3 \right)fdx_2 \wedge dx_4
        }
        \\
        &= \frac{1}{2\pi}\dd \paren[\bigg]{
            -\cos 2x_2\sin x_3 dx_3 \wedge dx_4-(3+2\sin 2x_2\cos x_3)
                dx_2 \wedge dx_4
        }
        \\*
        &= 0\,.
    \end{align*}
    So $A$ is \gls{ym}.
    The primitive component of the curvature, $F_p^A$ can be found by noting
    \[
        \Pi(\dx_1 \wedge \dx_2)
        = \frac{1}{2}\paren*{ \dx_1 \wedge \dx_2 - \dx_3 \wedge \dx_4 }\,.
    \]
    Therefore,
    \[
        F_p^A = 
        -\frac{1}{4\pi}\cos 2x_2\sin x_3
        \paren[\Big]{
            \dx_1\wedge \dx_2 - \dx_3\wedge \dx_4
        }
        + \frac{1}{2\pi} \left(1-\frac{1}{2}\sin 2x_2\cos x_3 \right)
            dx_1\wedge dx_3\,.
    \]
    The non-primitive component is determined by
    \[
        \Phi^A = -\frac{1}{4\pi}\cos 2x_2\sin x_3\,.
    \]
    Since $\Phi^A$ is not $d$-closed, $A$ is not \gls{tym} and also not \gls{pym}.
    By \cref{cor:Kahler:ym_iff_pym_for_G_abelian}, this implies that
    $(T^4, \omega, g, J)$ is not Kähler.
\end{ex}

\begin{ex}[A \gls{pym} solution on a closed manifold that is neither \gls{ym} nor \gls{tym}]
Starting from the solution of \cref{ex:T4_ym_not_pym_and_vice_versa}, we can use Hodge theory to prove the existence of a distinct connection on a $U(1)$-bundle over $T^4$ that is \gls{pym} but not \gls{ym}.
    The Bianchi identity \cref{eq:omega_bianchi} implies
    $\delm F_p^A = -\delp \Phi^A$.
    Differentiating both sides by $\delp$ implies $F_p$ is $\delp\delm$-closed.
    By the ellipticity of the symplectic Laplacian $\Delta_{++}$ in \cref{LaplacePP}, there exists some $\xi\in P^1(M)$ such that $F_p^A + \delp\xi$ is $\delp^*$-closed.
    Consider the connection $\tilde A = A + \xi$. Its curvature is
    ${\tilde F}^{\tilde A} = F^A+d\xi$, and the primitive part is exactly
    ${\tilde F_p}^{\tilde A} = F_p^A+\delp\xi$. Hence, $\tilde A$ is \gls{pym}.
    However, since $\delp^*F_p^A \neq 0$, $\delp\xi$ is not zero. So $\dd\xi\neq0$ and ${\tilde F}^{\tilde A} = F^A + \dd\xi$ is not $\dd$-harmonic.
    Thus, $\tilde A$ is not \gls{ym}, and by the \nameref{lem:2_out_of_3},
    $\tilde A$ is also not \gls{tym}.

\end{ex}

In the following example,
we show that the \gls{sd} BPST instantons
\cite{belavinPseudoparticleSolutionsYangMills1975}
are not \gls{pym}
(for a review of the BPST instanton, see
\cite{belitskyYangMillsDinstantons2000}).
\begin{ex}[\Glsfmtlong{sd} BPST Instanton: A \gls{ym} solution that is neither \gls{pym} nor \gls{tym}]\label{ex:BPST_instanton}
    Let $P$ be an $\SU(2)$-bundle over $M = \R^4$, and take the standard Euclidean metric and symplectic structure.
    Let
    \[
        A = \frac{2\eta^a_{\mu\nu} x^{\nu}}{|x|^2+1}T_a \dx^{\mu}\,,
    \]
    which is the connection of the SD  BPST instanton.
    Here, $\{ T_1,T_2,T_3 \}$ is a basis of $\mathfrak{su}(2)$ satisfying
    $[T_a,T_b] = \epsilon^{abc}T_c$, and
    $\eta^a_{\mu\nu}$ is the 't Hooft symbol
    $$
        \eta^a_{\mu\nu}=
        \begin{cases}
            \epsilon^{a\mu\nu}, & \mu,\nu=1,2,3 \\
            -\delta^{a\nu}, & \mu=4 \\
            \delta^{a\mu}, & \nu=4 \\
            0, & \mu=\nu=4
        \,.\end{cases}
    $$
    For example, we can set $T_a=\frac{1}{2i}\sigma_a$ for the Pauli matrices $\sigma_1,\sigma_2,\sigma_3$.
    Its curvature is
    \begin{equation*}
        F^A
        = \frac{4}{(|x|^2+1)^2}
        \paren[\Big]{
            - T_1 (dx^1\wedge dx^4+dx^2\wedge dx^3)
            + T_2 (dx^1\wedge dx^3-dx^2\wedge dx^4)
            - T_3\omega
        }\,.
    \end{equation*}
    With $F^A$ being self-dual, $A \in \critYM$.
    On the other hand,
    \[
        \Phi^A = -\frac{4}{(|x|^2+1)^2}T_3\,.
    \]
    Since $\frac{4}{(|x|^2+1)^2}$ is not $d$-closed, $d\Phi^A$ has a nontrivial $T_3$-component.
    Moreover, $[T_a,T_3]$ has no $T_3$-component for any $a$, so neither does
    $[A,\Phi^A]$. Therefore, $\dd_A\Phi^A$ has a nontrivial $T_3$-component and cannot vanish.
    It follows that $A$ is not \gls{tym}, and hence not \gls{pym}.

\end{ex}
This 4-dimensional example is especially interesting since self-duality is equivalent to the condition that $F_p^A$ has bidegree $(2, 0) + (0, 2)$.
We showed in \cref{cor:bidegree_equivalencies}(ii) that for $n \geq 3$, connections that satisfy this condition are \gls{ym} if and only if they are \gls{pym} and \gls{tym}.
Thus, the \gls{sd} BPST instanton is a counterexample to the extension of
\cref{cor:bidegree_equivalencies}(ii) to $n = 2$.

\begin{rmk}
    The construction of the \gls{sd} BPST instanton makes use of the homotopy classes of maps $S^3\rightarrow S^3$ from the boundary of $\R^4$ to $\SU(2) = S^3$ \cite{belitskyYangMillsDinstantons2000}. However, 
    the presence of $\om$ in the \gls{pym} conditions breaks the symmetry of $\R^4$ coordinates by separating them into two groupings of $(x^1, x^2)$ and $(x^3, x^4)$.
    For this reason, it is straightforward to extend
    \cref{ex:BPST_instanton} to show that the \gls{sd} BPST instanton is not \gls{pym} with respect to any symplectic form compatible with the Euclidean metric.
\end{rmk}

\subsection{\texorpdfstring{$\Omega$}{Omega}-Instantons}
\label{sec:omega_instantons}
For a 4-manifold, the anti-self-dual (ASD) equations $* F^A = -F^A$
are equivalent to $F^A = F_p^{1, 1}$.
In this case, $\Phi^A = 0$ and $\delma F_p^A = \dd_A \Phi^A = 0$.
Thus, \gls{asd} instantons are critical solutions of all three functionals.
The self-duality equations $* F^A = F^A$ are equivalent to
$F_p^A = F^{\calJ-}$.
In this case, since $n=2$, the first equation of
\cref{table:updated_table_for_critical_points}
is trivially satisfied, showing that \gls{sd} instantons are automatically Yang-Mills.
Since this does not constrain the curvature further, the other two equations need not be satisfied, and, in fact, are not in the case of the \gls{sd} BPST instanton in \cref{ex:BPST_instanton}.

For $d = 4$, the self-duality equations only constrain the primitive curvature $F_p^A = F^{\calJ-}$.
This is why \cref{cor:bidegree_equivalencies} has the $n = 2$ exception for
\gls{sd}-instantons. For higher-dimensional symplectic manifolds, the self-duality equations generalize and further constrain the non-primitive component of the curvature.
In this direction, let $\Omega \in \Omega^{d-4}(M)$ be an arbitrary form.
Then a connection $A$ is said to be $\Omega$-\gls{sd} ($\Omega$-\gls{asd}) if it satisfies the $\Omega$-instanton equations
\[
    * F^A = \pm F^A \wedge \Omega\,.
\]
See \cite{chenCompactness$Omega$YangMillsConnections2022} for an excellent introduction to $\Omega$-\gls{ym} connections and their compactness properties.
When $\Omega$ is closed, the $\Omega$-\gls{ym} equations\footnote{The $\Omega$-\gls{ym} equations for general $\Omega$ are $\dd_A^* (F^A + *(F^A \wedge \Omega)) = 0$.} and \gls{ym} equations are equivalent.

In the case of a symplectic manifold, there is a natural choice of a closed $\Omega$ which generalizes instantons on 4-manifolds:
$\Omega = \omega^{n-2}/(n-2)!$.
With this choice, the $\Omega$-instanton equations are
\begin{align*}
    \pm \frac{1}{(n-2)!} L^{n-2} F^A
    & %
    = * (F_p^A + \Phi^A \omega)
    = - \frac{1}{(n-2)!} L^{n-2}\calJ(F_p^A)
        + \frac{1}{(n-1)!} L^{n-1} \Phi^A
\,.\end{align*}
Equating Lefschetz components and using the injectivity of the Lefschetz operator, we find that a connection is $\Omega$-\gls{sd} if 
\[
    F =
    \begin{cases}
        F^{\calJ-} + \Phi^A\omega\,,
            & \text{if } n = 2
        \\
        F^{\calJ-}\,,
             & \text{if } n > 2
    \,.\end{cases}
\]
It is $\Omega$-\gls{asd} if
\[
    F^A= F_p^{1,1}\,.
\]
We have proven the following, which extends the analysis of the \gls{sd} BPST instanton case to higher dimensions.
\begin{prop}
    \label{prop:omega_instantons_are_triple_points}
    Let $(M^{2n}, \omega, g)$ be a symplectic manifold,
    and let $\Omega = \omega^{n-2}/(n-2)!$.
    Then
    \begin{enumerate}
        \item For $n \geq 3$, an $\Omega$-\gls{sd} connection is a critical solution of all three functionals.
        \item For $n \geq 2$, an $\Omega$-\gls{asd} connection is a critical solution of all three functionals.
    \end{enumerate}
\end{prop}

\subsection{Rigidity Results for \texorpdfstring{$\om$}{omega}-Flat Solutions }
\label{sec:rigidity_results}

The $\omega$-flat connections are \gls{pym} and zeros of the \gls{pym} functional.  Their curvature is given by $F^A = \Phi^A\omega$.  The Bianchi identity implies that $\Phi^A$ must be covariantly constant, and therefore, they are also \gls{tym}.  
Recall that covariantly constant sections of vector bundles have locally constant pointwise norms.
A more qualitative version of this result was discussed for $\omega$-flat connections in
\cite[Proposition~3.5]{tsengSymplecticFlatnessTwisted2022}, but we give a direct proof here that $\abs*{\Phi^A}$ is locally constant.
\begin{lem}
    \label{prop:prf_locally_constant_crit_pts}
    If $\Phi^A \in \Gamma(\ad P)$ satisfies $\dd_A \Phi^A = 0$, then
    $\abs*{\Phi^A}$ is locally constant.
\end{lem}
\begin{proof}

    Metric compatibility gives
    $\dd \abs*{\Phi^A}^2 = 2\innerp{\dd_A \Phi^A, \Phi^A} = 0.$
\end{proof}
This implies a particularly strong Liouville theorem for critical solutions of the \gls{tym} functional, and hence for $\omega$-flat connections.
This should be contrasted with the weaker Liouville theorem presented later in \cref{cor:liouville_theorem}.
\begin{cor}[$\omega$-flat Liouville]
\label{cor:symplectically_flat_liouville}
    Let $(M^d, \omega, g)$ be a connected symplectic manifold with a compatible metric and 
    $\vol(M) = \infty$.
    If $A$ is \gls{tym} and $\norm{\Phi^A\omega}^2 < \infty$, then $\Phi^A = 0$.
    In particular, there are no nontrivial $\omega$-flat connections with
    $\norm{F^A}^2 < \infty$ on manifolds with infinite volume.
\end{cor}

A similar rigidity statement holds on closed manifolds.
The following results are in the same spirit as
\cite[Example~3.1]{tsengSymplecticFlatnessTwisted2022}.
Recall that for an $\SO(r)$-bundle, the inner product on its Lie algebra is given by the negative Killing form
\[
    \innerp{u, v}_{\mathfrak{so}(r)} = -c \trace(uv)\,.
\]
Here, $c = c(r) > 0$ is some positive, often dimensional, constant.
The first Pontryagin class of $P$ for an $\SO(r)$-bundle, using the Atiyah-Bott convention of the inner product given in \cref{eq:convention:wedge_product},
\[
    p_1(P) = \frac{1}{8\pi^2c} F^A \wedge F^A\,,
\]
defines an integral cohomology class independent of the choice of connection.
Thus, the first Pontryagin number on a closed 4-manifold,
\[
    P_1(P) = \int_M p_1(P) \in \Z\,,
\]
is an integral topological invariant of the bundle.
For a $4k$-manifold, consider the Pontryagin number
\[
    P_{1^k}(P) \eqdef
    \int_M \underbrace{p_1(P) \wedge \cdots \wedge p_1(P)}_{k\text{ times}}
    \in \Z\,.
\]
Note that for an $\SO(r)$-bundle in our notation,
$\abs{\Phi^A}^2 = -c\trace((\Phi^A)^2) = \Phi^A \wedge \Phi^A$.
We are now ready for the next constraint on $\omega$-flat connections.
This generalizes the well-known fact that bundles admitting both \gls{sd} and \gls{asd} instantons must be flat.
\begin{prop}
    \label{prop:norm_of_Phi_for_omega_flat}
    Let $(M^{4k}, \omega)$ be a closed connected symplectic manifold.
    If $A$ is an $\omega$-flat connection on a principal $\SO(r)$-bundle, then
    \begin{equation}\label{eq:first_pont_number_constraint_on_omega_flat}
        \abs*{\Phi^A}^2
        = 8\pi^2c
            \paren*{ \frac{P_{1^k}}{(2k)!\vol(M)} }^{1/k}\,.
    \end{equation}
    In particular, a flat $\SO(r)$-bundle on a $4k$-dimensional closed symplectic manifold admits no non-flat $\omega$-flat connections.
\end{prop}
\begin{proof}
    Since $A$ is $\omega$-flat, $F = \Phi^A\omega$, and we can compute the Pontryagin number by
    \[
        P_{1^k}
        = \paren*{\frac{1}{8\pi^2c}}^k
            \int_M (\Phi^A\omega \wedge \Phi^A\omega)^k
        = \paren*{\frac{1}{8\pi^2c}}^k
            \int_M \abs*{\Phi^A}^{2k} \omega^{2k}\,.
    \]
    By \cref{prop:prf_locally_constant_crit_pts},
    \[
        P_{1^k}
        = \paren*{\frac{1}{8\pi^2c}}^k
        \abs*{\Phi^A}^{2k} (2k)!\vol(M)\,.
        \qedhere
    \]
\end{proof}
Similar algebraic constraints clearly hold for the other Pontryagin numbers by changing which class we choose to integrate.
The same techniques also apply verbatim to other compact matrix Lie groups.
This result also constrains the existence of $\omega$-flat connections according to the sign of $P_{1^k}$.
\begin{cor}
    Let $(M^{4k}, \omega)$ be a closed connected symplectic manifold.
    Let $P \to M$ be an $\SO(r)$-bundle.
    If $P_{1^k} < 0$, then there are no $\omega$-flat connections on $P$ with respect to any symplectic form $\omega \in \Omega^2(M)$.
\end{cor}
\begin{proof}
    For an $\omega$-flat connection to exist, the right-hand side of
    \cref{eq:first_pont_number_constraint_on_omega_flat} requires
    $P_{1^k} \geq 0$ since the left-hand side is nonnegative.
\end{proof}
\begin{ex}\label{ex:K3_surfaces}
    K3 surfaces are hyperkähler surfaces. Their tangent bundles are
    $\SO(4)$-bundles.
    The tangent bundle of every K3 surface $K$ has first Pontryagin class
    $p_1(K) = -48\ell$
    \cite[Section~5.1]{freedConsistencyMTheoryNonOrientable2021}, where $\ell \in H^4(K; \Z) \cong \Z$ is the positive generator.
    Thus $P_1(K) < 0$, and no K3 surface admits an $\omega$-flat connection on its tangent bundle for any of its symplectic forms.
\end{ex}

These results require either compactness or finite action.
As the following simple example shows, $M = \R^{2n}$ admits many $\omega$-flat connections.
\begin{ex}
    Let $M = \R^{2n}$ with the Euclidean metric and standard symplectic structure.    
    Consider a (necessarily trivial) $G$-bundle over $\R^{2n}$, 
    and let $A_0$ be the canonical flat connection.
    Let $\Phi^A \in \mathfrak{g}$ be fixed for some $A$, and
    let $\lambda \in \Omega^1(\R^{2n})$ be such that $\omega = \dd\lambda$.
    Then, the family of connections,
    \[
        A_t = A_0 + t\lambda \Phi^A\,,
    \]
    has curvature
    \[
        F^{A_t} = \dd(t \lambda \Phi^A)
            + \frac{1}{2}t^2\brack*{\lambda \Phi^A, \lambda \Phi^A}
        = t\Phi^A\omega\,.
    \]
    So we have a one-parameter family $A_t$, for $t \in \R$, of $\omega$-flat connections.
\end{ex}

\subsection{Classification of \texorpdfstring{$\omega$}{omega}-Flat Solutions and Relations to Solutions of the Cone Yang-Mills Functional}
\label{sec:classification_of_omega_flat}
We relate the \gls{pym} functional to the cone Yang-Mills functional that was recently introduced by Tseng and Zhou in \cite{tsengMappingConeConnections2025}.  %
To start, let us fix a closed two-form, $\zeta \in \Omega^2(M)$, that is not necessarily non-degenerate. Then, for any connection form $A$ and section $B \in \Gamma(\ad P)$, the cone Yang-Mills functional is defined to be
\[
    \CYM[\zeta](A, B) 
    \eqdef \int_M \paren*{ \abs*{F^A + \zeta B}^2 + \abs*{\dd_A B}^2 } \vol_g\,.
\]
A pair $(A, B)$ is called a cone-flat connection if it is a zero of this functional.
We will show that, in the special case $\zeta = \omega$, the zeros of the cone Yang-Mills functional coincide with the zeros of the \gls{pym} functional.
This lets us use the classification of cone-flat connections in \cite{tsengMappingConeConnections2025} 
to conclude a classification of $\om$-flat connections.

The cone-flat connections for $\zeta = \omega$ are pairs $(A, B)$ which satisfy the equations
\[
    F^A + \omega B = 0 \quad \text{and}\quad
    \dd_A B = 0\,.
\]
Lefschetz decomposing the curvature and comparing terms, we find from the first equation that $\Phi^A\omega = -B\omega$. Since
$L : \Omega^0(M, \ad P) \to \Omega^2(M, \ad P)$
is injective, $\Phi^A = -B$.
We have proven the following.
\begin{lem}
    The cone-flat connections with $\zeta = \omega$ are the pairs
    $(A, -\Phi^A)$, where $A$ is $\om$-flat.
\end{lem}
Thus, the cone Yang-Mills functional gives us another approach to studying the $\omega$-flat connections.
We can now use
\cite[Theorem~4.1]{tsengMappingConeConnections2025} which classifies cone-flat connections to immediately obtain the following classification of $\omega$-flat connections on all $G$-bundles.
\begin{prop}[Classification of $\omega$-flat connections]
\label{thm:classification_of_omega_flat_connections}
    Let $(M, \omega)$ be a connected symplectic manifold, and let $G$ be a Lie group.
    Then there is a bijective correspondence
    \[
        \correspondence
        {
            \text{isomorphism classes of $\omega$-flat}\\
            \text{connections on $G$-bundles over $M$}
        } {
            \text{conjugacy classes of}\\
            \text{homomorphisms $\rho : \Gamma \to G$}
        }\,,
    \]
    where $\Gamma$ is an $\R/\overline{H}$-extension of $\pi_1(M)$ and
    $\overline{H} \subseteq \R$ is the closure of the group
    \[
        H \eqdef \set*{
            \int_{S} \omega :
            [S] \in \pi_2(M)
        }\,.
    \]
\end{prop}

By the Lefschetz decomposition, we can decompose the functional further into
\[
    \CYM[\omega](A, B)
    = \int_M \paren*{
        \abs*{F^A_p}^2
        + \abs*{\omega(\Phi^A + B)}^2
        + \abs*{\dd_A B}^2 
    } \vol_g
    \geq \norm*{F^A_p}^2\,.
\]
This motivates the embedding of the space of connections into the space of possible $(A, B)$ pairs given by
\[
    A \mapsto (A, -\Phi^A)\,.
\]
The cone Yang-Mills functional for this pairing is
\[
    \CYM[\omega](A, -\Phi^A)
    = \int_M \paren*{ \abs*{F^A_p}^2 + \abs*{\dd_A \Phi^A}^2 } \vol_g\,.
\]

The pair $(A, B)$ is a critical solution of the cone Yang-Mills functional for general $\zeta$ if and only if
\begin{subequations}
\begin{align}
    \dd_A^*\paren*{ F^A + \zeta B } + \brack*{ B, \dd_A B } &= 0\,,
    \label{eq:cym_equation_1}
    \\
    \zeta^*\paren*{ F^A + \zeta B } + \dd_A^* \dd_A B &= 0\,.
    \label{eq:cym_equation_2}
\end{align}
\end{subequations}
Plugging $(A, -\Phi^A)$ into \cref{eq:cym_equation_2} with $\zeta = \omega$ gives
\[
    \omega^*\paren*{ F^A + \omega (-\Phi^A) } + \dd_A^* \dd_A (-\Phi^A)
    = \Lambda F^A_p - \laplace^A \Phi^A
    = - \laplace^A \Phi^A\,.
\]
This implies that $\Phi^A$ is harmonic if $(A, -\Phi^A)$ is a critical solution.
On a closed manifold, this is equivalent to $\dd_A \Phi^A = 0$, so $A$ would be a critical solution of the \gls{tym} functional.
\Cref{eq:cym_equation_1} can be written as
\[
    \dd_A^*\paren*{ F^A + \omega (-\Phi^A) } + \brack*{ (-\Phi^A), \dd_A (-\Phi^A) }
    = \dd_A^*F^A_p + \brack*{ \Phi^A, \dd_A \Phi^A }\,.
\]
We have proven the following:
\begin{prop}
    Let $(M, \omega, g)$ be a closed symplectic manifold. Then,
    $A$ is a critical solution of all three functionals,
    $\norm{F^A}^2$, $\norm{F^A_p}^2$, and $\norm{\Phi^A\omega}^2$,
    if and only if $(A, -\Phi^A)$ is a critical solution of
    $\CYM[\omega]$.
    Moreover, if $A$ is such a critical solution, then
    \[
        \norm*{F^A_p}^2 = \CYM[\omega](A, -\Phi^A)\,.
    \]
\end{prop}

\section{A Monotonicity Formula}
\label{sec:monotonicity_formula}
The goal of this section is to prove the following.
\MonotonicityFormula*
The monotonicity formula for the \gls{ym} functional, first proven by Price in \cite{priceMonotonicityFormulaYangMills1983}, plays a similar role to the monotonicity formula in harmonic map theory. Namely, it is a critical tool in the compactness theory of the moduli space of \gls{ym} connections.
The \gls{ym} monotonicity formula says that the scale-invariant $L^2$-norm of the curvature over a ball of radius $r$ is monotone increasing in $r$.
The formula for the \gls{pym} functional is similar; however, it carries the defect term $S_\tau$, as seen in \cref{eq:defect_monotonicity}. 

Our proof for the \gls{pym} monotonicity formula will be similar to that for the \gls{ym} case given in
\cite[Section~2.1]{tianGaugeTheoryCalibrated2000}.
Let $X \in \mathcal{X}(M)$ be a compactly supported vector field in $M$ (to be chosen later). Then $X$ generates a one-parameter family of diffeomorphisms
$\phi_t \in \Diff(M)$.
Fix a background connection $\tilde A$ over $M$. Then $\tilde A$ induces a parallel transport map $\tau_t : P \to P$, and we can lift the family of diffeomorphisms to $P$ as $\tilde \phi_t \in \Diff(P)$.
Given any connection $A$, it follows that $\tilde \phi_t$ generates a family of connections $A_t = \tilde\phi_t^*A$ with corresponding curvature forms
$F^{A_t} \in \Omega^2(M, \ad P)$ such that
\begin{equation}\label{eq:family_of_curvature_forms}
    F^{A_t}(V_1, V_2) = 
    \tau_t^{-1} \circ F^A(\pushf{\phi_t}V_1, \pushf{\phi_t}V_2)\,.
\end{equation}
Here, $\tau_t$ also denotes the induced parallel transport isometry on $\ad P$.

Using \cref{eq:family_of_curvature_forms}, we decompose the family of curvature forms into
\[
    \tau_t \circ F^{A_t}
    = \pullb{\phi_t}\paren*{ F_p^A + \Phi^A \omega }
    = \pullb{\phi_t}F_p^A + \paren*{ \pullb{\phi_t}\Phi^A }
        \paren*{ \pullb{\phi_t}\omega }\,.
\]
This is not the Lefschetz decomposition of $F^{A_t}$ anymore. All we can conclude in general is
\[
    F^{A_t}_p = \Pi F^{A_t}\,,\qquad
    \Phi^{A_t} = \frac{1}{n}\Lambda F^{A_t}\,.
\]
Since the monotonicity formula will make use of the first variation of the functional, we stop to consider
{\addtolength{\jot}{3pt}
\begin{align*}
    \frac{\dd}{\dd t}
        \abs*{F^{A_t}_p}^2
        \Big\rvert_{t=0}
    & %
    = \frac{\dd}{\dd t}
        \abs*{\Pi \circ \tau_t^{-1} \circ \pullb{\phi_t}F^{A}}^2
        \Big\rvert_{t=0}
    \\* &
    = \frac{\dd}{\dd t}
        \abs*{\tau_t^{-1} \circ \Pi \circ \pullb{\phi_t}F^{A}}^2
        \Big\rvert_{t=0}
       & \paren*{ \text{Lefschetz decomposition is on forms} }
    \\ &
    = \frac{\dd}{\dd t}
        \abs*{\Pi \circ \pullb{\phi_t}F^{A}}^2
        \Big\rvert_{t=0}
       & \paren*{ \text{Parallel transport is an isometry} }
    \\ &
    = 2\innerp*{
        \frac{\dd}{\dd t}
        \Pi \paren*{ \pullb{\phi_t} F^{A} }
        \Big\rvert_{t=0}
        , F_p^A
    }
\,.\end{align*}}
Since the symplectic form $\omega$ is fixed, $\Pi$ is time-independent. Together with its linearity,
\[
    \innerp*{
        \frac{\dd}{\dd t}
        \Pi \paren*{ \pullb{\phi_t} F^{A} }
        \Big\rvert_{t=0}
        , F_p^A
    }
    = \innerp*{
        \Pi
        \frac{\dd}{\dd t}
        \pullb{\phi_t} F^{A}
        \Big\rvert_{t=0}
        , F_p^A
    }
    = \innerp*{
        \frac{\dd}{\dd t}
        \pullb{\phi_t} F^{A}
        \Big\rvert_{t=0}
        , F_p^A
    }\,,
\]
where we dropped the projection operator in the last equality since the orthogonality of the Lefschetz decomposition ensures we only keep the primitive part.
Taking the derivative, we find
\[
    \frac{\dd}{\dd t}
        \pullb{\phi_t} F^{A}
        \Big\rvert_{t=0}
    = 
    \frac{\dd}{\dd t}
        \pullb{\phi_t} F^{A}_p
        \Big\rvert_{t=0}
    + (\calL_X \Phi^A)\omega
    + \Phi^A(\calL_X\omega)\,.
\]
We leave the first term as written for future convenience.
Taking the inner product with $F^A_p$, and applying orthogonality again, we obtain
\[
    \frac{\dd}{\dd t}
        \abs*{F^{A_t}_p}^2
        \Big\rvert_{t=0}
    = \frac{\dd}{\dd t}\abs*{\pullb{\phi_t}F^A_p}^2
    \Big\rvert_{t=0}
    + 2\innerp*{ \Phi^A(\calL_X\omega), F_p^A }\,.
\]

Let $(e_i)_{i=1}^d$ be a local orthonormal frame on $M$.
Then,
\[
    \norm*{\pullb{\phi_t}F^A_p}^2
    = \int_M \sum_{i,j}
        \abs*{ F_{p}^{A}
            \paren*{ \pushf{\phi_t}e_i(x), \pushf{\phi_t}e_j(x)} }^2
        \vol_g(x)\,.
\]
After a change of coordinates $x \mapsto \phi_{-t}(x)$, we find
\begin{equation}\label{eq:prim_curvature_on_family_of_connections}
    \norm*{\pullb{\phi_t}F^A_p}^2
    = \int_M \sum_{i,j}
    \abs*{ F^A_{p}
    \paren[\big]{
        \pushf{\phi_t}e_i(\phi_{-t}(x)),
        \pushf{\phi_t}e_j(\phi_{-t}(x))}
    }^2
    \phi_{-t}^*
    \vol_g(x)\,.
\end{equation}
\Cref{eq:prim_curvature_on_family_of_connections} is qualitatively identical to what we would find when expanding $\norm*{F^{A_t}}^2$ (see, for example,
\cite[Section~2.1]{tianGaugeTheoryCalibrated2000}).
Since $\norm*{F^{A_t}}^2$ is the quantity of interest when proving the monotonicity formula for the \gls{ym} functional, most of the manipulations of
\cref{eq:prim_curvature_on_family_of_connections} will mirror those for the
\gls{ym} case.

Our defect term is
\[
    S \eqdef \int_M \innerp*{ \Phi^A(\calL_X\omega), F_p^A }\vol_g
    = \Linnerp{\Phi^A\paren*{\calL_X \omega}, F_p^A}\,.
\]
The sign of $S$ will be a potential obstruction for a monotonicity formula for the primitive curvature.
The first variation is then
\begin{equation}\label{eq:first_variation_before_LC}
    \frac{\dd}{\dd t}
        \norm*{F_p^{A_t}}^2
    \Big\rvert_{t=0}
    = 2S - \int_M
    \paren[\bigg]{ 
    \abs*{ F^A_{p} }^2 \div X 
    + 4\sum_{i,j}^{} 
        \innerp*{
            F^A_{p} \paren*{ [X, e_i], e_j },
            F^A_{p} \paren*{ e_i, e_j }
        }
    } \vol_g\,.
\end{equation}

The remaining steps of the proof will now follow exactly those of the Yang-Mills case, except that we will be carrying the additional defect term throughout.
Hence, we will just present the arguments and refer the reader to \cite[Section~2.1]{tianGaugeTheoryCalibrated2000} and \cite[Section~3.2]{fadelIntroductionYangMillsTheory} for detailed derivations of each step.

Let $\nabla$ denote the Levi-Civita connection on $M$.
Using metric compatibility, we can rewrite \cref{eq:first_variation_before_LC} as 
\[
    \frac{\dd}{\dd t}
        \norm*{F_p^{A_t}}^2
    \Big\rvert_{t=0}
    = 2S - \int_{M} \paren[\bigg]{
    \abs{F^A_p}^2\div X - 4 \sum_{i,j}
        \innerp[\Big]{F^A_p(\nabla_{e_i}X, e_j), F^A_p(e_i, e_j)}
    } \vol_g\,.
\]
This is the \defterm{first variation formula} for the primitive Yang-Mills functional.
Now suppose $A = A_0$ is \gls{pym}.
By Cartan's formula, $\calL_X\omega = \dd\iota_X\omega$, where $\iota_X$ denotes interior multiplication.
Since $X$ is compactly supported, $\iota_X\omega$ is too.
Therefore,
\[
    0 = \Linnerp{\Phi^A\iota_X\omega, \dd_A^*F^A_p}
    = \Linnerp{\dd_A\Phi^A \wedge \iota_X\omega, F^A_p}
    + \Linnerp{\Phi^A\calL_X\omega, F^A_p}\,.
\]
So our defect term can be written as
\begin{equation}\label{eq:defect_term}
    S = \Linnerp{\iota_X\omega \wedge \dd_A\Phi^A, F^A_p}\,.
\end{equation}
Finally, the first variation reduces to
\[
    \int_{M} \paren[\bigg]{
    \abs{F^A_p}^2\div X - 4 \sum_{i,j}^{} 
        \innerp[\Big]{F^A_p(\nabla_{e_i}X, e_j), F^A_p(e_i, e_j)}
    } \vol_g
    = 2S\,.
\]

Fix $p \in M$. Denote the injectivity radius of $(M, g)$ at $p$ by
$\operatorname{inj}_{p}(M)$.
Let $0 < r(p) = r_g(p) \leq \operatorname{inj}_{p}(M)$ be such that there exist normal coordinates $(x^1, \ldots, x^d)$ on the geodesic ball
$B = B_{r(p)}(p)$, and a constant $c(p) = c_g(p)$, with the property that, in these coordinates, the metric satisfies
\begin{align*}
    \abs*{g_{ij} - \delta_{ij}} &\leq c(p)\abs*{x}^2\,,
    \\*
    \abs*{\del_k g_{ij}} &\leq c(p)\abs*{x}\,.
\end{align*}
We will assume, without loss of generality, that $r(p) \leq 1$. Then by standard scaling techniques
(see, e.g.,
\cite[Remarks~3.2.5 and~3.2.12]{fadelIntroductionYangMillsTheory}),
together with \cref{rmk:rescaling_the_pym_functional},
it will follow that the statement holds for arbitrary $g$.
Finally, if $M = \R^d$ equipped with the Euclidean metric, then we can take
$c(p) = 0$ and $r(p) = \infty$.

Let $\rho(x) = \sqrt{\sum_i(x^i)^2}$ be the radial distance function centered at $p$, and let $\xi$ be a smooth cutoff function supported in $[0, r(p)]$. We can now take
\begin{equation}\label{eq:price_vector_field}
    X(x) = \xi(\rho) \rho \del_\rho\,,
\end{equation}
where $\del_\rho$ denotes the radial vector field on the punctured ball.
Extend $\del_\rho$ to an orthonormal frame $(\del_\rho, e_2, \ldots, e_d)$.
Then, via estimates on the Hessian operator of $\rho$, we have the following:
\begin{center}
\begin{tblr}{
    colspec = { l l r },
    columns = {mode=math, cmd=\displaystyle},
}
    \nabla_{\del_\rho} \del_\rho = 0\,,\qquad
    & \nabla_{\rho e_i} \del_\rho = e_i + O(d)c(p)\rho^2\,,
    & i \geq 2\,, \\
    \nabla_{\del_\rho} X = (\xi'(\rho) \rho  + \xi(\rho))\del_\rho\,,\qquad\quad
    & \nabla_{e_i} X = \xi(\rho) \sum_{j=2}^{d} b_{ij}e_j\,, 
    & i \geq 2\,.
\end{tblr}
\end{center}
where $b_{ij} = \rho g(\nabla_{e_i}\del_\rho, e_j)$.
Standard estimates also imply
$\abs*{b_{ij} - \delta_{ij}} = O(d)c(p)\rho^2$.

Computing the divergence and rearranging gives
\[
    \int_{M} \abs{F^A_p}^2
        \paren[\big]{\xi'(\rho)\rho + (d-4)\xi(\rho) + O(d)c(p)\rho^2\xi(\rho)}
        \vol_g
    = 4\int_M \xi'(\rho)\rho \abs*{\iota_{\del_\rho}F^A_p}^2 \vol_g + 2S\,.
\]
Let $\eta$ be a bump function satisfying
\begin{align*}
    \eta(t) &= \begin{cases} 1 & t \in [0, 1] \\
     0  & t \in [1+\epsilon, \infty),\: \epsilon > 0 \end{cases}
    \\
    \eta'(t) &\leq ~0\,, \quad \forall \,t\,.
\end{align*}
Define the family of functions $\xi_\tau(\rho) \eqdef \eta(\rho / \tau)$.
Choose $\tau$ small enough and set
$\xi(\rho) = \xi_\tau(\rho)$.
From this we can replace
\[
    \rho \xi'(\rho) = -\tau \frac{\del}{\del \tau} \xi_\tau(\rho)\,.
\]
Now let
\[
    S_\tau \eqdef
    \int_M \xi_\tau(\rho) \rho \innerp*{
        \iota_{\del_\rho}\omega \wedge \dd_A \Phi, F^A_p
    }\vol_g\,.
\]
This lets us write
\begin{align*}
    \tau\frac{\del}{\del \tau}
    \int_{M} \xi_\tau(\rho) \abs{F^A_p}^2\vol_g
    + \int_{M} \xi_\tau(\rho) \abs{F^A_p}^2
        \paren[\Big]{(4-d) &+ O(d)c(p)\tau^2} \vol_g
    \\
    &=\: 4\tau\frac{\del}{\del \tau}\int_M 
    \xi_\tau(\rho) \abs*{\iota_{\del_\rho}F^A_p}^2
    \vol_g
    - 2S_\tau\,.
\end{align*}

Let $a \in \R$ be some real parameter, to be determined later.
Multiply the above by
$e^{a\tau^2}\tau^{3-d}$, and integrate with respect to $\tau$
from $0 < r_1 \leq r_2 < r(p)$ to obtain
\begin{multline*}
    \int_{r_1}^{r_2}
        \frac{\del}{\del\tau} \paren*{
            e^{a\tau^2}\tau^{4-d}
            \int_{M} \xi_\tau(\rho) \abs*{F^A_p}^2\vol_g
        } \dd\tau
    - 4\int_{r_1}^{r_2}
    e^{a\tau^2}\tau^{4-d} \paren*{
        \frac{\del}{\del \tau}
            \int_M \xi_\tau(\rho) \abs{\iota_{\del_\rho}F^A_p}^2 \vol_g
    }
    \dd\tau
    \\*
    =
    \int_{r_1}^{r_2}
    e^{a\tau^2}\tau^{5-d} \int_M
    \paren[\Big]{
        2a + O(d)c(p)
    } \xi_\tau(\rho)\abs{F^A_p}^2 \vol_g 
    \dd\tau
    - 2\int_{r_1}^{r_2}e^{a\tau^2}\tau^{3-d} S_\tau \dd\tau
\,.\end{multline*}
Choose $a$ such that $2a + O(d)c(p) \geq 0$ to drop this term altogether.
Sending $\epsilon \searrow 0$ proves \cref{thm:monotonicity_formula}.
As a corollary, we have the following Liouville-type theorem.
\begin{cor}[Primitive Liouville]
    \label{cor:liouville_theorem}
    Let $d = 2n > 4$.
    Let $A$ be a critical solution of all three functionals on a $G$-bundle over
    $\R^{d}$ with the standard Euclidean metric and symplectic form. If
    $\norm*{F^A_p}^2 < \infty$, then $A$ is $\omega$-flat.
\end{cor}
\begin{proof}
    Since $A$ is \gls{pym} and \gls{tym}, $S_\tau = 0$, and 
    for every $0 < r \leq R$, \cref{thm:monotonicity_formula} gives
    \[
        \frac{1}{r^{d-4}}\int_{B_0(r)} \abs*{F^A_p}^2 \dx
        \leq 
        \frac{1}{R^{d-4}}\int_{B_0(R)} \abs*{F^A_p}^2 \dx\,.
    \]
    Sending $R \to \infty$ together with finite action implies
    \[
        \frac{1}{r^{d-4}}\int_{B_0(r)} \abs*{F^A_p}^2 \dx \leq 0\,.
    \]
    Since $r$ was arbitrary, it follows that $F^A_p = 0$ and $A$ is $\omega$-flat.
\end{proof}

\end{document}